\documentclass[11pt,a4paper,reqno]{amsart}

\pdfoutput=1

\usepackage[utf8]{inputenc}
\usepackage[T1]{fontenc}
\usepackage{lmodern}
\usepackage{amsthm, amssymb, amsmath, amsfonts, mathrsfs}

\usepackage{cancel}

\usepackage{microtype}

\usepackage[colorlinks=true,pdfpagemode=none,urlcolor=blue,
linkcolor=blue,citecolor=blue]{hyperref}

\usepackage{relsize}

\usepackage{amsmath,amsfonts,amssymb,amsthm}
\usepackage{amsthm, amssymb, amsmath, amsfonts, mathrsfs}

\usepackage{mathrsfs,fge}
\usepackage{MnSymbol}

\usepackage{color}
\usepackage{accents}

\usepackage[normalem]{ulem}
\usepackage{bbm}

\usepackage{cleveref}

\definecolor{labelkey}{gray}{.8}
\definecolor{refkey}{gray}{.8}

\definecolor{darkred}{rgb}{0.9,0.1,0.1}
\definecolor{darkgreen}{rgb}{0,0.5,0}

\numberwithin{equation}{section}

\newcommand{\FF}{{\cal F}}

\newcommand{\HH}{{\cal H}}

\newcommand{\MM}{{\cal M}}

\newcommand{\PP}{{\cal P}}

\newcommand{\BR}{{\mathbb R}}

\newcommand{\sgn}{\mbox{\rm sgn}}
\newcommand{\sol}{{\mathcal S^{2}(0,T;\mathbb R^l)\times \mathcal M^{2}_0(0,T;\mathbb R^l)}}

\newcommand{\R}{{\mathbb R}}

\newcommand{\cal}{\mathcal}

\newtheorem{theorem}{\bf Theorem}[section]
\newtheorem{proposition}[theorem]{\bf Proposition}
\newtheorem{lemma}[theorem]{\bf Lemma}

\theoremstyle{definition}
\newtheorem{definition}[theorem]{Definition}
\newtheorem{example}[theorem]{\bf Example}

\newtheorem{remark}[theorem]{Remark}

\numberwithin{equation}{section}

\begin{document}
\title[Nonlinear BSDEs  in general filtration]{Nonlinear BSDEs  in general filtration with  drivers depending on the martingale part of a solution}
\author{Tomasz Klimsiak and Maurycy Rzymowski}
\address[Tomasz Klimsiak]{Institute of Mathematics, Polish Academy Of Sciences,
ul. \'{S}niadeckich 8,   00-656 Warsaw, Poland, \and Faculty of
Mathematics and Computer Science, Nicolaus Copernicus University,
Chopina 12/18, 87-100 Toruń, Poland, e-mail: {\tt tomas@mat.umk.pl}}

\address[Maurycy Rzymowski]{Faculty of
Mathematics and Computer Science, Nicolaus Copernicus University,
Chopina 12/18, 87-100 Toruń, Poland, e-mail: {\tt maurycyrzymowski@mat.umk.pl}}

\date{}
\maketitle
\begin{abstract}
In the present paper, we consider multidimensional  nonlinear backward stochastic differential equations (BSDEs) with  a driver depending on the martingale part $M$ of a solution. We assume that the nonlinear term is merely monotone continuous with respect to the state variable.
As to the regularity of the driver with respect to the  martingale variable, we consider a very  general condition which permits
path-dependence on "the future" of the process $M$ as well as  a dependence  of its law (McKean-Vlasov-type equations).
For such driver, we prove the existence and uniqueness  of  a {\em global solution} (i.e. for any maturity $T>0$) to BSDE with data  satisfying natural integrability conditions.
\end{abstract}

\section{Introduction}
\label{sec1}

\medskip
Consider  a filtered probability space  $(\Omega,\FF,P,\mathbb{F}:=(\mathcal{F}_t)_{t\ge 0})$,  maturity $T>0$,  $\mathcal{F}_T$-measurable  terminal condition  $\xi$, and a driver  
\[
f:\Omega\times [0,T]\times\mathbb{R}^l\times \mathcal M^{2}_0(0,T;\BR^l)\rightarrow\mathbb{R}^l
\]
which is $\mathbb F$-progressively measurable with respect to the first two variables. 
Here $\MM^{2}_0(0,T;\BR^l)$ denotes  the space of $l$-dimensional c\`adl\`ag $\mathbb F$-martingales  starting at zero  with  square integrable  supremum over $[0,T]$.  In the present paper, we study backward stochastic differential equations  (BSDEs) of the  form
\begin{equation}\label{eqi.1}
Y_t=\xi+\int^T_t f(r,Y_r,M)\,dr-\int^T_t\,dM_r,\quad t\in[0,T].
\end{equation}
A solution to \eqref{eqi.1} consists of an $l$-dimensional  $\mathbb F$-adapted c\`adl\`ag  process $Y$, and $M\in\MM^{2}_0(0,T;\BR^l)$
such that \eqref{eqi.1} holds $P$-a.s.  
We follow and generalize here the framework considered by Liang,  Lyons and Qian in \cite{LLQ}, and Bensoussan, Li and Yam in \cite{BLY}.  In these papers, the authors considered a special class of  \eqref{eqi.1} with a driver of the form
\begin{equation}\label{eqi.0}
f(r,Y_r,M)=g(r,Y_r,h(M)_r),
\end{equation}
where    $h$ is an operator, which maps $\MM^{2}_0(0,T;\BR^l)$ in the space  $\HH^2_{\mathbb F}(0,T;\mathbb R^m)$ - consisting of $\mathbb R^m$-valued $\mathbb F$-progressively measurable processes that are square integrable with respect to $dt\otimes dP$ - and  $g:[0,T]\times\Omega\times\mathbb{R}^l\times \mathbb R^m\rightarrow\mathbb{R}^l$.

The  purpose  of the present paper is to address the conjecture made in \cite{LLQ} about the existence
of a {\em global solution} to \eqref{eqi.1} with $f$ given by \eqref{eqi.0}   under the following condition on the operator $h$:  for some $\gamma_h> 0$, and any $t\in [0,T]$,
\begin{equation}
\label{eqi.2}
E\int_t^T|h(M)_r-h(N)_r|^2\,dr\le \gamma_h E\int_t^T \, d[M-N]_r,\quad  M,N\in   \MM^{2}_0(0,T;\BR^l).
\end{equation}
Under this condition, and Lipschitz continuity  of $g$
with respect to the second and the third variable, the authors proved in \cite{LLQ} that for small enough $T>0$, there exists a 
solution to \eqref{eqi.1} - the so called  {\em local solution}. For the existence of {\em global solutions} (i.e. for any $T>0$) $h$ is additionally  assumed in \cite{LLQ} to satisfy  {\em local-in-time property}:  for any $a<b,\, a,b\in [0,T]$, and any $M\in\MM^{2}_0(0,T;\BR^l)$,
\begin{equation}
\label{eqi.litp}
h(M)_t=h(M^{a,b})_t,\quad t\in (a,b)\quad\mbox{with} \quad M^{a,b}_t:=M_{(a\vee t)\wedge b},\quad t\in [0,T].
\end{equation}
The above condition stands that the evaluation of $h(M)_t$ depends only on values of $M_s,\, s\in [t,t+\varepsilon)$ for however small $\varepsilon>0$. In the paper, the authors conjectured (see the comments at the end of Section 3 in \cite{LLQ}) that  
regarding operator $h$, condition \eqref{eqi.2} alone is not sufficient  to guarantee the existence of a  {\em global solution} to \eqref{eqi.1}, and suggested that the following operator may serve as an example supporting their conjecture:
\begin{equation}
\label{eqi.3} 
h_1(M)_t=E\big(\int_t^T\,d[M]_r|\FF_t),\quad t\in [0,T].
\end{equation}
We see that the evaluation of $h_1(M)_t$ depends on the whole future, i.e. it depends on the values $M_s,\, s\in [t,T]$,
so it does not share {\em local-in-time property}.

Recently,  Cheridito  and  Nam in  \cite{CN},  considered even more general framework than presented above. They introduced the notion of backward stochastic equations (BSEs) of the form
\begin{equation}
\label{eqi.4}
Y_t+F_t(Y,M)+M_t=\xi+F_T(Y,M)+M_T,\quad t\in [0,T]
\end{equation}
for given generator $F$.
If we  take $F_t(Y,M)=\int_0^tf(r,Y_r,M)\,dr$, then  the above BSE becomes BSDE \eqref{eqi.1}. One of the results
obtained in \cite[Proposition 3.7]{CN} applies to BSDEs  \eqref{eqi.1} with $f$ given by \eqref{eqi.0}, and   $h$ of a special form. However, a  close inspection of the proof of \cite[Proposition 3.7]{CN} reveals that it in fact applies not only to $h$
specified therein but to any $h$ satisfying \eqref{eqi.2} and  the following additional condition:
\begin{equation}
\label{eqi.6}
\mbox{if}\,\,\, M,N\in\MM^{2}_0(0,T;\BR^l)\,\,\,\mbox{are strongly orthogonal, then}\,\,\, h(M+N)=h(M)+h(N).
\end{equation}
Therefore, from \cite[Proposition 3.7]{CN} it follows that  if $g$ is Lipschitz continuous with respect to the second and the third variable (uniformly in $t$), and $h$ satisfies \eqref{eqi.2},\eqref{eqi.6}, then there exists a {\em global solution} to \eqref{eqi.1}.
Observe that operator $h_1$ satisfies the latter condition. Thereby,  it does not support the conjecture stated in \cite{LLQ}.

The main result of the present paper (Theorem \ref{th3}) disproves the conjecture formulated in \cite{LLQ}, and shows that in fact  condition \eqref{eqi.2} and Lipschitz continuity of  driver $g$ with respect to the second and the third variable (uniformly in $t$)  guarantee the existence
of a unique {\em global solution} to \eqref{eqi.1} with $f$ of the form \eqref{eqi.0}. In fact, we shall prove our existence and uniqueness result of  {\em global solutions}  under considerable   weaker  hypotheses, and
in general framework \eqref{eqi.1}.   
We assume that $f$ is merely  continuous monotone in the second variable, i.e. for some $\mu\in\BR$, and any  $t\in [0,T]$, 
\[
\langle f(t,y_1,M)-f(t,y_2,M),y_1-y_2\rangle\le \mu |y_1-y_2|^2,\quad y_1,y_2\in \mathbb R^l,\, M\in \MM^{2}_0(0,T;\BR^l).
\]
Moreover, the growth of $f$ with respect to $y$ is subject to no restriction except the local behavior: for any $R>0$, 
\[
E\Big(\int_0^T\sup_{|y|\le R}|f(t,y,0)|\,dt\Big)^2<\infty.
\]
Instead of  \eqref{eqi.2}, which only fits into the framework \eqref{eqi.0}, we consider a more general condition
of the following form: for any $t\in [0,T]$, $y\in\mathbb R^l$,
\begin{equation}
\label{eqi.7}
E\int_t^T |f(r,y,M)-f(r,y,M')|^2\,dr\le\lambda E\int_t^T\,d[M-M']_r,\quad M, M'\in\MM^2_0(0,T;\BR^l),
\end{equation}
which holds in particular  when $f$ is given by \eqref{eqi.0}, $h$  satisfies \eqref{eqi.2}, and $g$ is Lipschitz continuous with respect to $h(M)_t$.
When $\mathbb F$ is generated by a Brownian motion $B$, $f$ is given by \eqref{eqi.0}, and $h(M)=Z^M$,
where $Z^M$ is a unique progressively measurable process such that $M_t=\int_0^tZ_r^M\,dB_r,\, t\in [0,T]$, then our assumptions
agree with the ones  considered in \cite{BDHPS} (with $p=2$ in \cite{BDHPS}; see also \cite[Section 6]{K:SPA}, where reflected BSDEs on general filtered spaces are considered). 

The organization of the paper is as follows. In Section \ref{sec2} we introduce the basic notation and hypotheses.
Section \ref{sec3} is devoted to presenting several examples of driver $f$ satisfying \eqref{eqi.7}.  
They  show that the approach and  framework 
considered in the present  paper provide a unified way of treating a wide variety of seemingly disparate classes of BSDEs.
The  definition of a solution and a priori estimates  are contained in Section \ref{sec4}. In Section \ref{sec5} we give an existence and uniqueness result for BSDEs with a driver independent of $M$.  Finally, Section \ref{sec6} contains the proof of the paper's main result. In Section \ref{sec7}, we demonstrate  that the proof technique used in Section \ref{sec6} may also be utilized to BSDEs
of the form
\begin{equation}\label{eqi.101}
Y_t=\xi+\int^T_t f(r,Y,M)\,dr-\int^T_t\,dM_r,\quad t\in[0,T],
\end{equation} 
where now the driver may also depend   on the future evaluation of $Y$. We prove
that under the following Lipschitz type condition:
\begin{enumerate}
\item[(A)] There exists $L>0$ such that for any $t\in [0,T]$, $M,M'\in \MM^2_0(0,T;\BR^l)$, $Y,Y'\in \mathcal S^2(0,T;\BR^l)$
\[
E\int_t^T|f(r,Y,M)-f(r,Y',M')|^2\,dr\le L(E\int_t^T\,d[M-M']_r+E\int_t^T|Y_r-Y'_r|^2\,dr),
\]
\end{enumerate}
and standard integrability assumptions on the data, there exists a unique {\em global solution} to \eqref{eqi.101}.

\section{Preliminaries}
\label{sec2}
\subsection{Basic notation}
Let $l\in\mathbb{N}$. For $x\in\mathbb{R}^l$, $|x|$ denotes the Euclidean norm and $\langle\cdot,\cdot\rangle$
the natural inner product in $\BR^l$. We set $\sgn(x)=\mathbf{1}_{|x|\neq 0}\,x/|x|$. Let $(\mathcal X,\rho)$ be a separable metric linear space. We denote by $L^{2}(\Omega,P;\mathcal X)$ the set of $\FF_T$-measurable $\mathcal X$-valued random variables $X$ such that
\[
E\rho^2(X,0)<\infty.
\]
We let
\[
\rho_{L^2}(X^1,X^2):= \big(E\rho^2(X^1,X^2)\big)^{1/2},
\]
and if $\mathcal X$ is a normed space, then we let $\|\cdot\|_{L^2}$ denote the norm generated by $\rho_{L^2}$.
By $\mathcal{S}^{2}_{\mathbb F}(0,T; \mathcal X)$, we denote the set of all $\mathcal X$-valued $\mathbb{F}$-progressively measurable processes $Y=(Y_t)_{t\in [0,T]}$, such that 
\[
E\sup_{0\le t\le T}\rho^2(Y_t,0)<\infty.
\]
We set
\[
\rho_{\mathcal{S}^{2}_{\mathbb F}}(Y^1,Y^2):= \big(E\sup_{0\le t\le T}\rho^2(Y^1_t,Y^2_t)\big)^{1/2}
\]
If $\mathcal X$ is a normed space, then by $\|\cdot\|_{\mathcal S^2_{\mathbb F}}$ we denote the norm generated by $\rho_{\mathcal S^2_{\mathbb F}}$.
 We let $\mathcal{M}_{loc}(0,T;\mathbb R^l)$  denote the set of all c\`adl\`ag processes  $M=(M^1_t,\dots,M^l)_{t\in[0,T]}$
 such that for any $i\in \{1,\dots,l\}$ process $M^i$ is a local $\mathbb{F}$-martingale.  By $\mathcal{M}^{2}(0,T;\BR^l)$ we denote the subset of $\mathcal{M}_{loc}(0,T;\mathbb R^l)$ consisting of processes $M$ such that 
 \[
\|M\|_{\MM^{2}}:=(E[M]_T)^{1/2}<\infty,
 \] 
where $[M]_T=\sum_{i=1}^l[M^i]_T$ and $[M^i]$ is the {\em square bracket} of $M^i,\, i=1,\dots,l$. 
By $\MM^{2}_{0,loc}(0,T;\mathbb R^l)$ (resp. $\MM^{2}_0(0,T;\BR^l)$) we denote the subspace of  $\MM^{2}_{loc}(0,T;\BR^l)$ (resp. $\MM^{2}(0,T;\BR^l)$)
consisting of those processes $M$ for which $M_0=0$. 
We let  $\mathcal{H}^2_{\mathbb F}(0,T;\mathcal X)$ be the set of all 
$\mathbb{F}$-progressively measurable 
 processes $X=(X_t)_{t\in[0,T]}$ taking values in $\mathcal X$, such that
\[
E\int^T_0\rho^2(X_r,0)\,dr<\infty.
\] 
We let
\[
\rho_{\HH^2_\mathbb F}(X^1,X^2):= \Big(E\int^T_0\rho^2(X^1_r,X^2_r)\,dr\Big)^{1/2}.
\]  
When $\mathcal X$ is a normed space, then by $\|\cdot\|_{\HH^2_{\mathbb F}}$ we denote the norm generated by $\rho_{\HH^2_{\mathbb F}}$.
By $B_{\mathbb F}(0,T; L^2(\Omega,P;\BR^l))$, we denote the set of all $\mathbb{F}$-adapted   $\mathbb R^l$-valued processes $V=(V_t)_{t\in[0,T]}$, such that
\[
\|V\|_{B_{\mathbb F}}:= \sup_{0\le t\le T}\big(E|V_t|^2\big)^{1/2}<\infty.
\]

In the whole paper all relations between  random variables hold
$P$-a.s. For $\mathbb{R}^l$-valued processes $X$, $Y$ we  write  $X_t=Y_t,\, t\in [0,T]$ (or simply $X=Y$) iff 
\[
P(\exists_{t\in [0,T]}\,\, X_t\neq Y_t)=0.
\]
For given $d\in\mathbb N$ we let $\ell^d$ denote $d$-dimensional Lebesgue measure on $\BR^d$.

\subsection{Assumptions on the data}

The generator (driver) is the mapping
\[
f: \Omega\times[0,T]\times\mathbb{R}^l\times \MM^{2}_0(0,T;\BR^l) \to\mathbb{R}^l,
\]
which is $\mathbb{F}$-adapted for  fixed $y\in\BR^l, M\in\MM^2_0(0,T;\BR^l)$.

We shall need the following hypotheses:
\begin{enumerate}
\item[(H1)]$E|\xi|^2<\infty,\, E(\int_0^T|f(r,0,0)|\,dr)^2<\infty$,
\item[(H2)] There exists  $\mu\in\R$ such that  for any $t\in[0,T]$, $y,y'\in\R^l$, $M\in\MM^{2}_0(0,T;\BR^l)$,
\[
\langle y-y',f(t,y,M)-f(t,y',M)\rangle \leq\mu|y-y'|^2,
\]
\item[(H3)] There exists  $\lambda >0$ such that for any $t\in[0,T]$, $y\in\R^l$, $M,M'\in \MM^{2}_0(0,T;\BR^l)$,
\[
E\int_t^T |f(r,y,M)-f(r,y,M')|^2\,dr\le\lambda E\int_t^T\,d[M-M']_r,
\]
\item[(H4)] For every $(t,M)\in[0,T]\times \MM^{2}_0(0,T;\BR^l)$ the
mapping $\mathbb{R}^l\ni y\rightarrow f(t,y,M)$ is continuous $P$-a.s.,
\item[(H5)] For each $r>0$, 
\[
E\big(\int^T_0 \psi_r(t)\,dt\big)^2<\infty,
\]
where $\psi_r(t)=\sup_{|y|<r}|f(t,y,0)-f(t,0,0)|$.
\end{enumerate}

\section{Examples of generators satisfying (H3)}\label{sec3}

In this section we shall give several examples of generators satisfying (H3). They  show that the approach and  framework 
considered in the present  paper provides a unified way of treating a wide variety of seemingly disparate classes of BSDEs.
We start with  generators of the form \eqref{eqi.0} (Examples \ref{ex.1}--\ref{ex.5a}). Then BSDE \eqref{eqi.1} is of the following form 
\begin{equation}\label{eqi.1a}
Y_t=\xi+\int^T_t g(r,Y_r,h(M)_r)\,dr-\int^T_t\,dM_r,\quad t\in[0,T].
\end{equation}
In fact we consider slightly  more general case, when $\BR^m$ is replaced by the metric space $\mathcal X$.
In this case Lipschitz continuity of $g$ with respect to the third variable and condition \eqref{eqi.2}, which now admits the following form
\begin{equation}
\label{eqi.2a}
E\int_t^T\rho^2(h(M)_r,h(N)_r)\,dr\le \gamma_h E\int_t^T \, d[M-N]_r,\quad  M,N\in   \MM^{2}_0(0,T;\BR^l),
\end{equation}
imply (H3). In Example \ref{ex.6} we show that more general than \eqref{eqi.1a} framework \eqref{eqi.1} allows us in particular to
cover a class of McKean-Vlasov-type equations.

\begin{example}[Classical BSDEs]
\label{ex.1}
In the pioneering paper \cite{PP} Pardoux and Peng introduced the notion of BSDEs, with a filtration $\mathbb F$
that was assumed to be generated by a given $d$-dimensional Brownian motion $B$, of the following form
\begin{equation}
\label{eq.ex1}
Y_t=\xi+\int_t^Tf(r,Y_r,Z_r)\,dr-\int_t^T Z_r\,dB_r,\quad t\in [0,T].
\end{equation}
A solution of equation \eqref{eq.ex1} consists of a continuous  process $Y\in \mathcal S^2_{\mathbb F}(0,T;\BR^l)$, and 
$Z\in \HH^2_{\mathbb F}(0,T; \BR^l\otimes\BR^d)$. Observe that BSDE \eqref{eqi.1a} becomes BSDE \eqref{eq.ex1}
if we consider $\mathbb F$ as specified above, $\mathcal X=\BR^d$ and 
\[
h(M)=Z^M,
\]
where $Z^M$
is a unique process in $\HH^2_{\mathbb F}(0,T; \BR^l\otimes\BR^d)$ such that $M_t=\int_0^t Z^M_r\,dB_r,\, t\in [0,T]$.
Clearly, $h$ satisfies \eqref{eqi.2a}.
\end{example}

\begin{example}[L\'evy processes and Teugels martingales]
\label{ex.2}
In \cite{NS} Nualart and Schoutens considered filtration $\mathbb F$ generated by a $d$-dimensional L\'evy process $(X_t)_{t\ge 0}$
with the L\'evy triplet $(a,\sigma,\nu)$. They  assumed additionally  that for some $\alpha>0$ and  any $\varepsilon>0$,
\[
\int_{\BR\setminus(-\varepsilon,\varepsilon)}e^{\alpha|x|}\,\nu(dx)<\infty.
\]
Then there exists an orthonormal  sequence $(H^i)_{i\ge 0}\subset \MM^2(0,T;\BR^l)$, built up 
from   Teugels martingales, that furnishes the following representation property: for any  $M\in\MM_0^2(0,T;\BR^l)$ there exists a unique sequence $(Z^{M,i})_{i\ge 1}\subset  \HH^2_{\mathbb F}(0,T;\BR^l\otimes\BR^d)$ such that
\[
M_t=\sum_{i=1}^\infty \int_0^t Z^{M,i}_r\,dH^i_r,\quad t\in [0,T]. 
\]
Based on this representation the authors in \cite{NS} introduced and studied  BSDEs of the  form
\begin{equation}
\label{eq.ex2}
Y_t=\xi+\int_t^Tf(r,Y_r,(Z^i_r)_{i\ge 1})\,dr-\sum_{i=1}^\infty\int_t^T Z^i_r\,dH^i_r,\quad t\in [0,T].
\end{equation}
Observe that equation \eqref{eqi.1a} becomes \eqref{eq.ex2} if we take $\mathcal X=l^2:=\{(a_n)_{n\ge 1}: \|(a_n)_{n\ge 1}\|^2_{l^2}:=\sum_{n=1}^\infty a_n^2<\infty\}$, and 
\[
h(M)=(Z^{M,i})_{i\ge 1}.
\]
One easily checks that $h$ fulfills \eqref{eqi.2a}.
\end{example}

\begin{remark}
A more advanced  model of the above type on general filtered spaces was considered in \cite[Section 6]{K:SPA}
for reflected BSDEs (see also \cite{CE} for BSDEs).
\end{remark}

\begin{example}[Poisson random measure]
\label{ex.3}
Assume that $B$ is a $d$-dimensional Brownian motion with respect to $\mathbb F$ and $N$ is a $d$-dimensional Poisson 
random measure on $E:=\BR^m\setminus\{0\}$ (independent of $B$), with an intensity
measure $\ell^1\otimes \nu$ ($\nu=(\nu_1,\dots,\nu_d)$) satisfying
\[
\sum_{i=1}^d\int_E1\wedge |x|^2\,\nu_i(dx)<\infty,
\]
such that for any $B\in\mathcal B(E)$, $\tilde N([0,t],A):=N([0,t],A)-t\nu(A)$ is an $\BR^d$-valued martingale with respect to $\mathbb F$.
In this case any martingale $M\in\MM^2_0(0,T;\BR^l)$ admits a unique representation (see e.g. \cite[Section III.4]{JS})
\[
M_t=\int_0^tZ^M_r\,dB_r+\int_0^t\int_E U_r(x)\,\tilde N(dr,dx)+L^M_t,\quad t\in [0,T], 
\]
where $Z^M\in\HH^2_{\mathbb F}(0,T; \BR^l\otimes\BR^d)$, $U\in \HH^2_{\mathbb F}(0,T; L^2(E,\nu;\BR^l\otimes\BR^d))$, and $L^M\in\MM^2_0(0,T;\BR^l)$
is strongly orthogonal to $B$ and $\tilde N$. There are numerous results in the literature (see e.g. \cite{BBP,KP,QS,Royer,Situ,YM}) concerning BSDEs of the form
\begin{equation}
\label{eq.ex3}
Y_t=\xi+\int_t^Tf(r,Y_r,Z_r,U_r)\,dr-\int_t^TZ_r\,dB_r-\int_t^T\int_EU_r(x)\,\tilde N(dr,dx),\quad t\in [0,T].
\end{equation}
This equation fits into our framework again. To see this take $\mathcal X=\BR^d\times L^2(E,\nu;\BR^l\otimes\BR^d)$
with natural metric $\rho$, and
\begin{equation}
\label{eq.ex.rep}
h(M)=(Z^M,U^M).
\end{equation}
One easily checks that \eqref{eqi.2a} holds.
\end{example}

\begin{example}[Carr\'e du champs operators]
\label{ex.4}
One of the major advantage  of BSDE's theory is that it provides an easy path for proving
probabilistic interpretation of solutions to a class of semilinear PDEs.
Let $\mathbb F$ be  the filtration generated by a Markov process $(X_t)_{t\ge 0}$, say Feller process, associated  to
Feller generator $(A,D(A))$ on $\BR^d$ with $C_c^\infty(\BR^d)\subset D(A)$. It is well known that any such operator is of the form
\begin{align*}
Au(x)=\sum_{i,j=1}^d q_{i,j}(x)u_{x_i x_j}(x)&+\sum_{i=1}^d b_i(x)u_{x_i}(x)+c(x)u(x)\\&
+\int_{\BR^d}\big(u(x+y)-u(x)-\mathbf{1}_{\{|y|\le 1\}}\langle \nabla u(x),y\rangle\big)\,N(x,dy),
\end{align*}
with $x$-dependent L\'evy triplet $(b(x),Q(x), N(x,dy))$, and $c\le 0$ (see e.g. \cite{BSW}).
When $(X_t)_{t\ge 0}$ is a diffusion process, i.e. $N\equiv 0$, then solutions to the Cauchy problem
\begin{equation} 
\label{eq.ex4}
-\partial_t u-Au=g(t,x,u,\sigma\nabla u),\quad u(T,\cdot)=\varphi,
\end{equation}
where $\sigma\cdot \sigma^T=Q$,
are related to BSDEs of the form \eqref{eq.ex1} with $\xi=\varphi(X_T)$ and $f(r,y,z)=\hat f(r,X_r,y,z)$ (see e.g. \cite{BPS,Pardoux,K:JFA,Pardoux1,R:PTRF}).
However, for general operator $A$ it is more natural to consider $\sqrt{\Gamma(u)}$ in place of $\sigma\nabla u$ in the nonlinear part
of equation \eqref{eq.ex4}, where $\Gamma$ is the so called  {\bf carr\'e du champs operator}:
\[
\Gamma(u):= A(u^2)-2uAu,\quad u\in D(A).
\]
Solutions of such equations are related to  BSDE \eqref{eqi.1a}, where $\xi,f$ are  as in the foregoing, and
\[
h(M)_r=\sqrt{\frac{d\langle M\rangle_r}{dr}}
\]
(see e.g. \cite{BR}). Observe that \eqref{eqi.2a} is satisfied once again.
\end{example}

In all the above examples operator $h$ fulfills {\em local-in-time property} \eqref{eqi.litp}. Furthermore, it satisfies condition \eqref{eqi.6}. Now, we shall give  simple examples of operator $h$  which fulfills \eqref{eqi.2}, and at the same time  does not satisfy \eqref{eqi.litp} nor \eqref{eqi.6}. In the sequel we denote by $D(0,T;\BR^l)$ the set of all c\`adl\`ag $\BR^l$-valued functions on $[0,T]$. %

\begin{example}[Beyond local-in-time property; path dependent BSDEs]
\label{ex.5}
Consider a functional  $\phi:D(0,T;\BR^l)\rightarrow \BR$ such that for some $L>0$ and any $a,b\in D(0,T;\BR^l)$,
\[
|\phi(a)-\phi(b)|\le L\sup_{t\in [0,T]}|a(t)-b(t)|.
\]
Then operator $h:\MM^2_0(0,T;\BR^l)\rightarrow \HH^2_{\mathbb F}(0,T;\BR)$, defined by
\[
h(M)_t=E(\phi(M_{\cdot\wedge t}-M_t)|\FF_t),
\]
satisfies \eqref{eqi.2}, however  in general it  does not satisfy \eqref{eqi.litp} nor \eqref{eqi.6} (take e.g. $\phi(a)=\sup_{t\in [0,T]}|a(t)|$).
\end{example}

\begin{example}[Beyond local-in-time property; anticipated BSDEs]
\label{ex.5a}
Let $\zeta,\delta:[0,T]\rightarrow \BR^+$ be continuous functions. With  the notation of Example \ref{ex.3}, we set
\[
h(M)_t=(E(Z^M_{t+\zeta(t)}|\FF_t),E(U^M_{t+\delta(t)}|\FF_t)),\quad t\in [0,T],
\]
where $(Z^M_t,U^M_t)=(\eta_t,\hat\eta_t),\, t\ge T$ for given stochastic processes $\eta,\hat\eta$.
BSDEs of the form  \eqref{eqi.1a} with  $h$ as above were introduced (for Brownian filtration) by Peng and Yang in \cite{PY}.
Observe that $h$ satisfies \eqref{eqi.2a} and does not satisfy \eqref{eqi.litp}. If we consider in the definition of $h$
nonlinear expectation (see \cite{Peng}), then condition \eqref{eqi.6} will be also violated. 
\end{example}

The last example exhibits the advantage of the general framework \eqref{eqi.1} over the framework \eqref{eqi.1a}.
By $\mathcal P_1(\mathcal X)$ we denote the set of all probability measures $\mu$ on $\mathcal X$ for which
\[
\int_{\mathcal X}\rho(x,0)\,\mu(dx)<\infty.
\]
We let $W_{\mathcal X}$ denote the {\em Wasserstein metric} on $\mathcal P_1(\mathcal X)$:
\begin{align*}
W_{\mathcal X}(\mu,\nu)=\inf\{\int_{\mathcal X\times \mathcal X}\rho(x,y)\,P(dx,dy): P\in\PP_{\mu,\nu}(\mathcal X\times\mathcal X)\},
\end{align*}
where $\PP_{\mu,\nu}(\mathcal X\times\mathcal X)=\{P\in \PP_1(\mathcal X)\times \PP_1(\mathcal X):  P(dx,\mathcal X)=\mu(dx),\, P(\mathcal X,dy)=\nu(dy)\}$.

\begin{example}[McKean-Vlasov-type equations]
\label{ex.6}
Let $\mathcal X=D(0,T;\BR^l)$, and $\rho$ be the Skorokhod metric on $D(0,T;\BR^l)$.
Consider a function 
\[
\hat f:\Omega\times [0,T]\times \BR^l\times\PP_1(\mathcal X)\to\BR^l
\]
such that for some $L>0$ and any $t\in [0,T]$, $y\in\BR^l$,
\[
|\hat f(t,y,\mu)-\hat f(t,y,\nu)|\le L W_{\mathcal X}(\mu,\nu),\quad \mu,\nu \in\PP_1(\mathcal X).
\]
For given martingale $M\in\MM^{2}_0(0,T;\BR^l)$, and $t\in [0,T]$ we denote $M^t=M_{\cdot\wedge t}-M_t$.
We set
\[
f(t,y,M):=\hat f(t,y,\mathcal L(M^t)),\quad t\in [0,T], y\in\BR^l, M\in\MM^2_0(0,T;\BR^l).
\]
Observe that
\[
|f(r,y,M)-f(r,y,N)|\le LW_{\mathcal X}(\mathcal L(M^r),\mathcal L(N^r))\le LE\rho(M^r,N^r)\le LE\sup_{s\in [r,T]}|M^r_s-N^r_s|.
\]
Therefore, by the Burkholder-Davis-Gundy inequality, we  easily obtain (H3). We may also get  an interesting example 
of generator $f$ by taking $\mathcal X= (\BR^l\otimes\BR^d)\times L^2(E,\nu;\BR^l)$ (cf. \cite{BLP,CN}), and natural metric $\rho$
on $\mathcal X$ (under assumptions and notation of Example \ref{ex.3}).
Then, by using representation \eqref{eq.ex.rep}, we may define
\[
f(t,y,M):=\hat f(t,y,\mathcal L(Z^M_t,U^M_t)),\quad t\in [0,T], y\in\BR^l, M\in\MM^2_0(0,T;\BR^l).
\]
One easily checks that $f$ satisfies (H3).
\end{example}

\section{Definition of a solution to BSDE and a priori estimates}
\label{sec4}

\begin{definition}\label{def1}
We say that a pair $(Y,M)$ of $\mathbb{F}$-adapted processes is a solution of the backward stochastic differential equation with right-hand side $f$ and terminal value $\xi$ (BSDE($\xi,f$) for short) if
\begin{enumerate}
\item[(a)] $Y$ is c\`adl\`ag and $M\in\mathcal{M}^{2}_0(0,T;\BR^l)$,
\item[(b)] $\int^T_0|f(r,Y_r,M)|\,dr<+\infty$,
\item[(c)] $Y_t=\xi+\int^T_t f(r,Y_r,M)\,dr-\int^T_t\,dM_r,\quad t\in[0,T].$
\end{enumerate}
\end{definition}

\begin{proposition}\label{prop1}
Assume that \textnormal{(H1)--(H3)} are in force. Then there exists $C>0$, depending only on $T,\mu^+,\lambda$, such that  for any solution $(Y,M)\in\sol$ to  \textnormal{BSDE}$(\xi,f)$,
\begin{equation*}
E\sup_{0\le t\le T}|Y_t|^2+E\int^T_0 \,d[M]_r\le CE\Bigg(|\xi|^2+\Big(\int^T_0 |f(r,0,0)|\,dr\Big)^2\Bigg).
\end{equation*}
\end{proposition}
\begin{proof}
All the constants in the proof labelled  $C_i$ for some $i\in\mathbb N$, shall  depend only on $\lambda,\mu^+,T$.
By It\^o's formula
\begin{equation}\label{prop1.1}
|Y_t|^2+\int^T_t \,d[M]_r= |\xi|^2+2\int^T_t \langle Y_r, f(r,Y_r,M)\rangle\,dr-2\int^T_t \langle Y_{r-}, dM_r\rangle. 
\end{equation}
By (H2)
\begin{equation}\label{prop1.2}
\begin{split}
&\langle Y_r, f(r,Y_r,M)\rangle\le   |Y_r||f(r,Y_r,M)-f(r,Y_r,0)|+\mu^+|Y_r|^2+|Y_r||f(r,0,0)|\\
&\quad\le \frac12(1+\nu)|Y_r|^2+\frac{1}{2(1+\nu)}|f(r,Y_r,M)-f(r,Y_r,0)|^2+\mu^+|Y_r|^2+|Y_r||f(r,0,0)|,
\end{split}
\end{equation}
for every $\nu\ge 0$. 
From   (\ref{prop1.1}), (\ref{prop1.2}) we get
\begin{equation}
\label{eq3.waog12}
\begin{split}
|Y_t|^2&+\int^T_t \,d[M]_r\le |\xi|^2+(1+\nu+2\mu^+)\int^T_t |Y_r|^2\,dr\\&
+\frac{1}{1+\nu}\int^T_t |f(r,Y_r,M)-f(r,Y_r,0)|^2\,dr\\&+2\int^T_t |Y_r||f(r,0,0)|\,dr-2\int^T_t \langle Y_{r-}, dM_r\rangle. 
\end{split}
\end{equation}
Since $Y\in\mathcal S^{2}(0,T;\BR^l)$, and $M\in\MM^{2}_0(0,T;\BR^l)$, we have that $\int_0^\cdot \langle Y_{r-}, dM_r\rangle$
is a martingale. 
Thus, taking the expectation, and using (H3), we obtain
\begin{equation}\label{prop1.4}
\begin{split}
E|Y_t|^2&+E\int^T_t \,d[M]_r\le E|\xi|^2+(1+\nu+2\mu^+)E\int^T_t |Y_r|^2\,dr\\
&+\frac{\lambda}{1+\nu}E\int^T_t \,d[M]_r+2E\int^T_t |Y_r||f(r,0,0)|\,dr.
\end{split}
\end{equation}
Hence, for  $\nu:=2\lambda$, we get
\begin{equation}\label{prop1.4waog}
\begin{split}
E|Y_t|^2+\frac12E\int^T_t \,d[M]_r&\le E|\xi|^2+(1+\nu+2\mu^+)E\int^T_t |Y_r|^2\,dr\\&\quad +2E\int^T_0 |Y_r||f(r,0,0)|\,dr.
\end{split}
\end{equation}
Set $\beta_\nu:= (1+\nu+2\mu^+)$. By the Gronwall lemma
\[
E|Y_t|^2\le e^{\beta_\nu T}(E|\xi|^2+2E\int^T_0 |Y_r||f(r,0,0)|\,dr).
\]
This combined with (\ref{prop1.4waog}) yields
\begin{equation}\label{prop1.5}
\begin{split}
E|Y_t|^2+E\int^T_0 \,d[M]_r\le C_1EX,
\end{split}
\end{equation}
for some $C_1>0$, where $X:=|\xi|^2+2\int^T_0 |Y_r||f(r,0,0)|\,dr$. From (\ref{eq3.waog12}), (H3), (\ref{prop1.5}), and the Burkholder-Davis-Gundy inequality,
we conclude that for any $\beta>0$,
\begin{equation}\label{prop1.6}
\begin{split}
E\sup_{0\le t\le T}|Y_t|^2&\le \frac{\lambda}{1+\nu}E\int_0^T\,d[M]_r\,dr+C_2EX+2E\sup_{0\le t\le T}\Big|\int^T_t \langle Y_{r-}, dM_r\rangle\Big|\\&\le \frac{\lambda}{1+\nu}E\int_0^T\,d[M]_r\,dr+ C_2EX+C_3E\sup_{0\le t\le T}|Y_t|\left(\int^T_0 \,d[M]_r\right)^{\frac{1}{2}}\\
&\le (\frac12+C_3\beta^{-1})E\int_0^T\,d[M]_r\,dr+
C_2EX+C_3\beta E\sup_{0\le t\le T}|Y_t|^2\\&\le \big[(\frac12+C_3\beta^{-1})C_1+C_2\big] EX+C_3\beta E\sup_{0\le t\le T}|Y_t|^2.
\end{split} 
\end{equation}
Taking $\beta=\frac{1}{2C_3}$ we get that for some $C_4>0$,
\begin{equation}\label{prop1.7}
\begin{split}
E\sup_{0\le t\le T}|Y_t|^2+E\int^T_0 \,d[M]_r\le C_4EX
\end{split}
\end{equation}
Observe that for any $\alpha>0$,
\begin{equation*}
EX\le \alpha E\sup_{0\le t\le T}|Y_t|^2+\alpha^{-1}E\Bigg(|\xi|^2+\Big(\int^T_0 |f(r,0,0)|\,dr\Big)^2\Bigg).
\end{equation*}
Taking $\alpha=\frac{1}{2C_4}$ we conclude from  (\ref{prop1.7}) the desired inequality.
\end{proof}

When $l=1$ we are able to say something more about integrability of the driver $f$.

\begin{proposition}\label{th4}
Suppose that $l=1$. Assume that \textnormal{(H1)--(H3)}  are in force. Then there exists $C>0$, depending only on $T,\mu^+,\lambda$, such that  for any solution $(Y,M)\in\sol$ to  \textnormal{BSDE}$(\xi,f)$,
\[
E\Big(\int^T_0 |f(r,Y_r,M)|\,dr\Big)^2\le CE\Bigg(|\xi|^2+\Big(\int^T_0|f(r,0,0)|\,dr\Big)^2\Bigg).
\] 
\end{proposition}
\begin{proof}
Throughout  the proof  $C$ denotes a constant which can vary from line to line but depends only on $T,\mu^+,\lambda$.
 By the It\^o-Meyer formula 
\begin{equation}\label{th4.1}
\begin{split}
|Y_t|\le |\xi|+\int^T_t \sgn(Y_r) f(r,Y_r,M)\,dr-\int^T_t \sgn(Y_{r-})\,dM_r,\quad t\in [0,T].
\end{split} 
 \end{equation}
 Hence
 \begin{equation}\label{th4.1bnm}
\begin{split}
|Y_t|\le |\xi|+\int^T_t \sgn(Y_r) f(r,Y_r,M)\,dr+\sup_{0\le t\le T}\Big|\int^T_t\sgn(Y_{r-})\,dM_r\Big|,\quad t\in [0,T].
\end{split} 
 \end{equation}
By (H2) 
 \begin{equation*}
 \begin{split}
-\sgn(Y_r) f(r,Y_r,M)&\ge -\sgn(Y_r)(f(r,Y_r,0)-f(r,0,0))+\mu^+|Y_r|\\&\quad-|f(r,Y_r,M)-f(r,Y_r,0)|-|f(r,0,0)|-\mu^+|Y_r|\\&
=\big|\sgn(Y_r) (f(r,Y_r,0)-f(r,0,0))-\mu^+|Y_r|\big|\\&\quad
-|f(r,Y_r,M)-f(r,Y_r,0)|-|f(r,0,0)|-\mu^+|Y_r|\\&\ge
\big|f(r,Y_r,0)-f(r,0,0)\big|-\mu^+|Y_r|\\&\quad
-|f(r,Y_r,M)-f(r,Y_r,0)|-|f(r,0,0)|-\mu^+|Y_r|\\&\ge
|f(r,Y_r,0)|-|f(r,Y_r,M)-f(r,Y_r,0)|-2|f(r,0,0)|-2\mu^+|Y_r|.
 \end{split}
 \end{equation*}
 From this inequality and  (\ref{th4.1bnm}) we infer that
 \begin{equation}\label{th4.2}
 \begin{split}
&\Big(\int^T_0|f(r,Y_r,0)|\,dr\Big)^2\le C\Bigg(|\xi|^2+\int^T_0|f(r,Y_r,M)-f(r,Y_r,0)|^2\,dr\\&\quad\quad
+\Big(\int^T_0|f(r,0,0)|\,dr\Big)^2+\int_0^T|Y_r|\,dr+\sup_{0\le t\le T}\Big|\int^T_t\sgn(Y_{r-})\,dM_r\Big|^2\Bigg).
 \end{split}
 \end{equation}
 Thus, applying the expectation to (\ref{th4.2}), and using  (H3) and Burkholder-Davis-Gundy inequality yields
 \begin{equation}\label{th4.4}
 \begin{split}
E\Big(\int^T_0|f(r,Y_r,0)|\,dr\Big)^2\le CE\Big(|\xi|^2+\Big(\int^T_0|f(r,0,0)|\,dr\Big)^2+\int_0^T|Y_r|\,dr+\int^T_0 d[M]_r\Big).
 \end{split}
\end{equation}
Therefore,  by  (H3)  and Proposition \ref{prop1} we get the desired inequality.
\end{proof}

\section{BSDEs with  driver independent of the martingale part }
\label{sec5}

In the present section we shall give an existence result for BSDEs in case the driver $f$
does not depend on the martingale part of a solution. The difficulty lies in the fact that
$f$ is assumed only monotone with respect to $y$ with  no assumptions on the growth. 
In the proof, among others,  we adapt and combine the techniques  applied in \cite{Pardoux1} and  \cite{BDHPS}, where the authors considered   BSDEs 
on Brownian filtration.

\begin{proposition}\label{Proposition3}
Assume that $f$ does not depend on $M$ and \textnormal{(H1),(H2),(H4),(H5)} are in force. Then there exists $(Y,M)$ - a solution to BSDE$(\xi,f)$ - such that $Y\in\mathcal{S}^{2}(0,T;\BR^l)$.
\end{proposition}
\begin{proof}
Since $f$ does not depend on $M$, (H5) is of the form
\begin{enumerate}
\item[(H5)] for any $r>0$, $E(\int^T_0 \psi_r(t)\,dt)^2<\infty$, where $\psi_r(t)=\sup_{|y|<r}|f(t,y)-f(t,0)|$.
\end{enumerate}
Without loss of generality, we may assume that $\mu\le 0$. We divide the proof into three steps.\\
\textbf{Step 1.} In the first step, we need the following additional assumptions: there exist $\alpha, \beta>0$ such that 
\begin{equation}\label{eq5.1}
|f(t,y)|\le \alpha |f(t,0)|+\beta,\quad t\in[0,T],\,y\in\mathbb{R}^l,
\end{equation}
and  there exist $C>0$,  such that
\begin{equation}
\label{eq5.2}
|\xi|+\sup_{0\le t\le T}|f(t,0)|\le C
\end{equation}
In particular, the above conditions  imply  that $|f|$ is bounded. Let $\varrho$ be a non-negative smooth function on $\BR^l$ with support in $B(0,1):=\{x\in\BR^l: |x|\le 1\}$ such that $\varrho(0)=1$, and $\varrho(x)<1,\, x\neq 0$.
We set 
\[
f_n(t,y):=\int_{\BR^l} f(t,x)\varrho_n(y-x)\,dx,\quad t\in [0,T], y\in \BR^l,
\]
where $\varrho_n(x)=a_n\rho(nx)$, $a_n>0$, and 
\[
\int_{\mathbb{R}^l}\varrho_n(x)\,dx=1,\quad n\ge 1.
\]
It is well known  that $f_n$ is smooth and  $f_n\rightarrow f$ uniformly on compacts (both properties with respect to $y$ under fixed $\omega\in\Omega$ and $t\in [0,T]$).
Observe that by assumptions made in the present step
\[
|f_n(t,y)-f_n(t,y')|\le 2\ell^l(B(0,1))(\alpha C+\beta)\|\nabla \varrho_n\|_\infty|y-y'|,\quad t\in [0,T],\, y,y'\in \BR^l,
\]
where $\ell^l$ is $l$-dimensional Lebesgue measure. Furthermore,  $f_n$ satisfies (H2) with $\mu\le0$, and 
\begin{equation}\label{eq5.3}
|f_n(t,y)|\le (\alpha C+\beta).
\end{equation}
By \cite[Theorem 4.1]{LLQ} the problem BSDE$(\xi,f_n)$ admits a unique solution  $(Y^n,M^n)\in \sol $.
By Ito's formula
\begin{equation}\label{eq5.4}
\begin{split}
e^t|Y^n_t|^2+\int^T_t e^r |Y^n_r|^2\,dr+\int^T_te^r\,d[M^n]_r&=e^T|\xi|^2+2\int^T_t e^r \langle Y^n_r, f_n(r,Y^n_r)\rangle\,dr\\
&\quad-2\int^T_r e^r\langle Y^n_{r-}, dM^n_r\rangle,\quad t\in [0,T].
\end{split}
\end{equation}
Using  (H2) (cf. the comment preceding \eqref{eq5.3}), we get
\begin{equation*}
2\langle Y^n_r, f_n(r,Y^n_r)\rangle \le |Y^n_r|^2+|f_n(t,0)|^2.
\end{equation*}
Therefore, applying conditional expectation to both sides of \eqref{eq5.4} yields
\[
e^t|Y^n_t|^2\le E\Big(e^T|\xi|^2+\int^T_0 e^r|f_n(r,0)|^2\,dr|\mathcal{F}_t\Big).
\]
From this,  \eqref{eq5.2} and \eqref{eq5.3},
\begin{equation}\label{eq5.5}
|Y^n_t|\le e^T[C+\sqrt T(\alpha C+\beta)],\quad t\in [0,T].
\end{equation}
Let $U^n_t=f_n(t,Y^n_t)$, $t\in[0,T]$. By \eqref{eq5.3}, Proposition \ref{prop1},  and the 
Burkholder-Davis-Gundy inequality,
\[
\sup_{n\in\mathbb{N}}\Bigg(E\int^T_0 |U^n_t|^2\,dt+E\int^T_0\,d[M^n]_r+E\sup_{t\le T}|M^n_t|^2\Bigg)<\infty.
\]
Since both $\HH^{2}_{\mathbb F}(0,T;\BR^l)$, $L^{2}(\Omega,P;\BR^l)$ are   Hilbert spaces, they are  reflexive. Therefore, by the Banach-Alaoglu theorem there exists a
subsequence (not relabelled) such that   $(U^n,M^n)$  converges weakly in the space $\HH^{2}_{\mathbb F}(0,T;\BR^l)$ to $(U,\tilde M)$ and $M^n_T$ converges weakly in the space $L^{2}(\Omega,P;\BR^l)$ to $N$. Let $M$ be a c\`adl\`ag version of  the martingale $E(N|\mathcal{F}_t)$, $t\in[0,T]$. 
We shall prove that for any  $\tau\in\mathcal{T}$ 
\begin{equation}\label{eq5.7}
M^n_{\tau}\rightarrow M_{\tau},\quad \mbox{weakly in }\quad L^{2}(\Omega,P;\BR^l).
\end{equation}
Let $X\in L^{2}(\Omega,P;\BR^l)$, then
\[
E\langle M^n_{\tau},X\rangle=E\langle E(M^n_T|\mathcal{F}_{\tau}), X\rangle=E\langle M^n_T, E(X|\mathcal{F}_{\tau})\rangle \rightarrow E\langle M_T, E(X|\mathcal{F}_{\tau})\rangle=E\langle M_{\tau},X\rangle.
\]
Observe that $\tilde M=M$ in $\HH^{2}_{\mathbb F}(0,T;\BR^l)$.
Indeed, let  $Z\in \HH^{2}_{\mathbb F}(0,T;\BR^l)$. Then $Z_r\in L^{2}(\Omega,P;\BR^l)$ for a.e. $r\in [0,T]$. 
By \eqref{eq5.7}, 
\[
E\langle M^n_T,Z_r\rangle\to E\langle M_T,Z_r\rangle,\quad\mbox{for a.e. }\,\, r\in [0,T].
\]
Therefore,
\begin{align*}
E\int_0^T\langle \tilde M_r, Z_r\rangle \,dr&=\lim_{n\to\infty} E\int_0^T\langle M^n_r,Z_r\rangle\,dr\\&
=\lim_{n\to\infty}\int_0^TE\langle M^n_T,Z_r\rangle\,dr=\int_0^TE\langle M_T,Z_r\rangle\,dr=E\int_0^T\langle M_r,Z_r\rangle\,dr.
\end{align*}
In the third equation we applied  the Lebesgue dominated convergence theorem by using the following estimation
\[
E\langle M^n_T,Z_r\rangle \le \sup_{n\ge 1} E\sup_{t\le T}|M^n_t|^2+E|Z_r|^2.
\]
The next convergence property we need is as follows:  for any $t\in [0,T]$
\begin{equation}\label{eq5.8}
\int^T_{t} U^n_r\,dr\rightarrow\int^T_{t} U_r\,dr\quad\mbox{ weakly in }\quad L^{2}(\Omega,P;\BR^l).
\end{equation}
This easily follows  from the following calculation.
Let  $X\in L^{2}(\Omega,P;\BR^l)$, then
\begin{equation*}\label{eq5.8}
\begin{split}
E \langle X,\int^T_t U^n_r\,dr\rangle&=E\int^T_t \langle X, E(U^n_r|\mathcal{F}_r)\rangle \,dr\\
&\quad=E\int^T_0 \langle U^n_r, \mathbf{1}_{[t,T]}(r)E(X|\mathcal{F}_r)\rangle\,dr\rightarrow E\int^T_0 \langle U_r, \mathbf{1}_{[t,T]}(r)E(X|\mathcal{F}_r)\rangle\,dr\\&
\quad=E \langle X,\int^T_t U_r\,dr\rangle.
\end{split}  
\end{equation*}
Set
\[
Y_t:=\xi+\int^T_t U_r\,dr-\int^T_t\,dM_r,\quad t\in[0,T].
\]
By  \eqref{eq5.7} and \eqref{eq5.8} for any $t\in [0,T]$
\begin{equation}
\label{eq5.9}
Y^n_t\rightarrow Y_t\quad\mbox{ weakly in }\quad L^{2}(\Omega,P;\BR^l).
\end{equation}
We shall prove  that $U=f(\cdot,Y)$ in $\HH^{2}_{\mathbb F}(0,T;\BR^l)$. Let $X\in \HH^{2}_{\mathbb F}(0,T;\BR^l)$. By \eqref{eq5.3}, we have  $f_n(\cdot,X)\rightarrow f(\cdot,X)$ in $\HH^{2}_{\mathbb F}(0,T;\BR^l)$. 
From this and (H2) we conclude
\begin{equation}\label{eq5.10}
\limsup_{n\rightarrow\infty}E\int^T_0\langle Y^n_t-X_t, f_n(t,Y^n_t)-f(t,X_t)\rangle\,dt\le 0.
\end{equation}
By Ito's formula
\begin{equation}
\label{eq5.11}
2E\int^T_0 \langle Y^n_r,f_n(r,Y^n_r)\rangle\,dr=E|Y^n_0|^2-E|\xi|^2+\int^T_0\,d[M^n]_r.
\end{equation}
Now, we apply the very well known fact saying  that in Hilbert spaces the norm generated by the inner product is weakly lower semicontinuous.
Thus, by \eqref{eq5.9}
\[
\liminf_{n\to\infty}E|Y^n_0|^2\ge E|Y_0|^2,
\]
and by \eqref{eq5.7} 
\begin{align*}
\liminf_{n\rightarrow \infty}E\int_0^T\,d[M^n]_r=\liminf_{n\to\infty}E|M^n_T|^2\,dr
\ge E|M_T|^2\,dr=E\int_0^T\,d[M]_r.
\end{align*}
The above two inequalities when applied to \eqref{eq5.11} give 
\begin{equation}
\label{eq5.12}
\liminf_{n\rightarrow\infty} E\int^T_0 \langle Y^n_t, f(t,Y^n_t)\rangle\,dt\ge E|Y_0|^2-E|\xi|^2+\int^T_0\,d[M]_r=E\int^T_0 \langle Y_t,U_t\rangle\,dt.
\end{equation}
It follows from this and  \eqref{eq5.10} that
\begin{equation}\label{eq5.13}
\begin{split}
&E\int^T_0\langle Y_t-X_t,U_t-f(t,X_t)\rangle\,dt\\
&\quad\le \liminf_{n\rightarrow\infty} E\int^T_0\langle Y^n_t-X_t, f_n(t,Y^n_t)-f(t,X_t)\rangle\,dt\le 0. 
\end{split}
\end{equation}
Let us choose $X_t=Y_t-\varepsilon(U_t-f(t,Y_t))$, $t\in[0,T]$, $\varepsilon>0$. Dividing  \eqref{eq5.13}  by $\varepsilon$ and letting $\varepsilon\rightarrow 0$, we obtain
\[
E\int^T_0|U_t-f(t,Y_t)|^2\,dt\le 0.
\]
This concludes the proof of Step 1.

\textbf{Step 2.} 
In this step we dispense with \eqref{eq5.1} but we shall keep in force \eqref{eq5.2}.
Let $a>0$ be such that
\[
e^{T/2}(|\xi|+\sqrt{T}\sup_{0\le t\le T}|f(t,0)|)\le a.
\]
Let $\theta_a$ be a smooth function on $\BR^l$ such that $0\le \theta_a\le 1$, $\theta_a(y)=1$ for $|y|\le a$ and $\theta_a(y)=0$ for $|y|\ge a+1$. For $n\in\mathbb{N}$ let us consider (cf. (H5))
\[
\hat f_n(t,y)=\theta_a(y)\cdot (f(t,y)-f(t,0))\cdot\frac{n}{\psi_{a+1}(t)\vee n}+f(t,0).
\]
For each $n\in\mathbb{N}$, $\hat f_n$ satisfies (H1), (H4), (H5). Furthermore, $\hat f_n$ also satisfies (H2) but with the positive constant $\mu$.
Indeed, let us take $t\in[0,T]$, $y,y'\in\mathbb{R}^l$ and assume that $|y|\ge a+1$ and $|y'|\ge a+1$. Then the inequality in  (H2) is trivially satisfied with any $\mu\ge 0$ since $\theta_a(y)=\theta_a(y')=0$. Suppose that  $|y|< a+1$. Since $f$ satisfies (H2) with $\mu\le 0$ we have
\begin{equation*}
\begin{split}
\langle y-y',\hat f_n(t,y)-\hat f_n(t,y')\rangle\le \frac{n(\theta_a(y)-\theta_a(y'))}{\psi_{a+1}(t)\vee n}\langle y-y',f(t,y)-f(t,0)\rangle.
\end{split}
\end{equation*}
By the very definition of $\psi_{a+1}$, we have $|f(t,y)-f(t,0)|\le \psi_{a+1}(t)$. Therefore 
\begin{equation*}
\begin{split}
\langle y-y',\hat f_n(t,y)-\hat f_n(t,y')\rangle\le n\|\nabla\theta_a\|_\infty|y-y'|^2.
\end{split}
\end{equation*}
Thus, $\hat f_n$ satisfies  (H2)  with $\mu=n\|\nabla \theta_a\|_\infty$. Furthermore, since  $\theta_{a}(y)|f(t,y)-f(t,0)|\le \psi_{a+1}(t),\, t\in [0,T],\, y\in\BR^l$,
we also have
\[
|\hat f_n(t,y)|\le n+|f(t,0)|,\quad t\in[0,T],\,y\in\mathbb{R}^l.
\]
Therefore, by Step 1, the problem BSDE$(\xi,\hat f_n)$ admits a unique solution $(Y^n,M^n)\in\sol$.
By Ito's formula
\begin{equation}\label{eq5.14}
\begin{split}
&e^t|Y^n_t|^2+\int^T_t e^r\,d[M^n]_r+\int^T_t e^r|Y^n_r|^2\,dr= e^T|\xi|^2+2\int^T_t e^r\langle Y^n_r, \hat f_n(r,Y^n_r)\rangle\,dr\\
&\quad-2\int^T_t e^r\langle Y^n_{r-}, dM^n_r\rangle,\quad t\in[0,T].
\end{split}
\end{equation}
Since $f$ satisfies (H2) with $\mu\le 0$, we have 
\[
2\langle Y^n_t, \hat f_n(t,Y^n_t)\rangle\le 2\langle Y^n_t, f(t,0)\rangle\le |Y^n_t|^2+\sup_{0\le t\le T}|f(t,0)|^2.
\]
Therefore, applying the conditional expectation to  \eqref{eq5.14}, we get
\[
\sup_{0\le t\le T}|Y^n_t|^2\le e^{T}E(|\xi|^2+T\sup_{0\le t\le T}|f(t,0)|^2|\FF_t)\le a^2.
\]
As a result, we obtain that in fact $(Y^n,M^n)$ is a solution to BSDE$(\xi,f_n)$, where
\[
f_n(t,y)=(f(t,y)-f(t,0))\frac{n}{\psi_{a+1}(t)\vee n}+f(t,0),\quad t\in[0,T],\,y\in\mathbb{R}^l.
\]
Note that $f_n$ satisfies (H2) with $\mu\le 0$.
Set $U=Y^{n+m}-Y^n$ and $V=M^{n+m}-M^n$. By Ito's formula
\begin{equation}\label{eq5.15}
\begin{split}
|U_t|^2+\int^T_t\,d[V]_r&=2\int^T_t\langle U_r, f_{n+m}(r,Y^{n+m}_r)-f_n(r,Y^n_r)\rangle\,dr\\&\quad
-2\int^T_t\langle U_{r-}, dV_r\rangle,\quad t\in [0,T]. 
\end{split}
\end{equation}
By (H2) (applied to $f_{n+m}$) and  the fact that $\sup_{0\le t\le T}|U_t|\le 2a$, we have
\begin{equation}\label{eq5.16}
\langle U_r, f_{n+m}(r,Y^{n+m}_r)-f_n(r,Y^n_r)\rangle\le 2a|f_{n+m}(r,Y^{n}_r)-f_n(r,Y^n_r)|.
\end{equation}
From this  and  \eqref{eq5.15}, we conclude
\begin{equation}\label{eq5.17}
E\int^T_0\,d[V]_r\le 4aE\int^T_0 |f_{n+m}(r,Y^n_r)-f_n(r,Y^n_r)|\,dr.
\end{equation}
By \eqref{eq5.15}, \eqref{eq5.16} 
\begin{equation}
\label{eq5.18}
E\sup_{0\le t\le T}|U_t|^2\le 4aE\int^T_0 |f_{n+m}(r,Y^n_r)-f_n(r,Y^n_r)|\,dr+E\sup_{0\le t\le T}\Bigg|2\int^T_t\langle U_{r-}, dV_r\rangle\Bigg|.
\end{equation}
By  the  Burkholder-Davis-Gundy inequality and Young's inequality there exists $C_1>0$ such that 
\begin{equation}
\label{eq5.19}
\begin{split}
E\sup_{0\le t\le T}\Bigg|2\int^T_t\langle U_{r-}, dV_r\rangle\Bigg|&\le C_1 E\Bigg(\int^T_0|U_r|^2\,d[V]_r\Bigg)^{\frac{1}{2}}\le C_1 E\sup_{0\le t\le T}|U_r|\Bigg(\int^T_0\,d[V]_r\Bigg)^{\frac{1}{2}}\\&\le \frac{1}{2}E\sup_{0\le t\le T}|U_t|^2+C^2_1 E\int^T_0\,d[V]_r.
\end{split}
\end{equation}
By  \eqref{eq5.17}--\eqref{eq5.19}
\begin{equation*}
E\sup_{0\le t\le T}|U_t|^2+E\int^T_0\,d[V]_r\le 8a(C^2_1+1)E\int^T_0 |f_{n+m}(r,Y^n_r)-f_n(r,Y^n_r)|\,dr.
\end{equation*}
Observe  that
\[
|f_{n+m}(t,Y^n_t)-f_n(t,Y^n_t)|\le 2\psi_{a}(t)\mathbf{1}_{\{\psi_{a}(t)>n\}}.
\]
Therefore, using (H5), we conclude that  $(Y^n,M^n)$ is a Cauchy sequence in the space $\sol$.
Thus, there exists $(Y,M)\in\sol$ such that
\begin{equation}
\label{eq5.19a}
E\sup_{0\le t\le T}|Y^n_t-Y_t|^2+E\int^T_0\,d[M^n-M]_r\rightarrow 0,\quad n\rightarrow\infty.
\end{equation}
By the definition of a solution to BSDE$(\xi,f_n)$,
\begin{equation}
\label{eq5.19b}
Y^n_t=\xi+\int^T_t f_n(r,Y^n_r)\,dr-\int^T_t\,dM^n_r,\quad t\in[0,T].
\end{equation}
By \eqref{eq5.19a}, up to subsequence, $Y^n_t(\omega)\rightarrow Y_t(\omega)$ for $\ell^1\otimes P$-a.e. $(t,\omega)\in[0,T]\times \Omega$.
Thus, by (H4), $f_n(\cdot,Y^n)\rightarrow f(\cdot,Y),\, \ell^1\otimes P$-a.e. 
Furthermore, since $|Y^n_t|\le a,\, t\in [0,T],\, n\ge 1$, we have that for any $t\in [0,T]$,
\begin{equation}
\label{eq5.19c}
|f_n(t,Y^n_t)|\le \psi_a(t)+|f(t,0)|,\quad n\ge 1.
\end{equation}
Therefore, from   (H5), and the Lebesgue dominated convergence theorem, we infer that
\begin{equation}
\label{eq5.19l}
\sup_{t\le T}\Big|\int_t^T f_n(r,Y^n_r)\,dr-\int_t^Tf(r,Y_r)\,dr\Big|\to 0,\quad n\to\infty.
\end{equation}
Thanks to  this convergence and \eqref{eq5.19a}, we may pass to the limit in \eqref{eq5.19b} getting that  $(Y,M)$ is a solution to  BSDE$(\xi,f)$.\\

\textbf{Step 3.}
Finally, we shall dispense with \eqref{eq5.2}. For each $n\in\mathbb{N}$ set
\[
\xi_n:=\Pi_n(\xi),\quad f_n(t,y):=f(t,y)-f(t,0)+\Pi_n(f(t,0)),
\]
where $\Pi_n(y)=y\min(n,|y|)|y|^{-1}$.
Obviously,
\begin{equation}\label{eq5.20}
E|\xi-\xi_n|^2\rightarrow 0,\quad E\Bigg(\int^T_0|f(t,0)-f_n(t,0)|\,dt\Bigg)^2\rightarrow 0,
\end{equation}
as $n\rightarrow\infty$.  We see that  for each $n\in\mathbb{N}$, the pair $(\xi_n,f_n)$ satisfies the assumptions of Step 2.  Therefore,  for any $n\in\mathbb{N}$, there exists a unique solution  $(Y^n,M^n)\in\sol$ to BSDE$(\xi_n,f_n)$. 
Let
\[
F(r,y):= f_n(r,y+Y^m_r)-f_m(r,Y^m_r),\quad r\in [0,T],\, y\in\BR^l.
\]
Observe that $(Y^n-Y^m,M^n-M^m)$ is a solution to BSDE$(\xi_n-\xi_m,F)$, and data $\xi_n-\xi_m,F$ satisfy (H1)--(H2) (in place of $\xi$ and $f$, respectively).
Thus, by Proposition \ref{prop1}
\begin{equation}
\label{eq5.21}
\begin{split}
E\sup_{0\le t\le T}|Y^n_t-Y^m_t|^2&+E\int^T_0\,d[M^n-M^m]_r\le C\Big(E|\xi^n-\xi^m|^2+E\big(\int_0^T|F(r,0)|\,dr\big)^2\Big)\\&
\quad=C\Bigg(E|\xi_n-\xi_m|^2+\Big(\int^T_0|f_n(r,0)-f_m(r,0)|\,dr\Big)^2\Bigg).
\end{split}
\end{equation}
From this and  \eqref{eq5.20}, we conclude that
\[
E\sup_{0\le t\le T}|Y^n_t-Y^m_t|^2+E\int^T_0\,d[M^n-M^m]_r\rightarrow 0,\quad n,m\rightarrow\infty.
\]
This means that $(Y^n,M^n)$ is a Cauchy sequence in $\sol$. 
Thus, there exists $(Y,M)\in\sol $ such that
\begin{equation}
\label{eq5.22}
E\sup_{0\le t\le T}|Y^n_t-Y_t|^2+E\int^T_0\,d[M^n-M]_r\rightarrow 0,\quad n\rightarrow\infty.
\end{equation}
By the definition of a solution to BSDE$(\xi_n,f_n)$,
\begin{equation}
\label{eq5.23}
Y^n_t=\xi_n+\int^T_t f_n(r,Y^n_r)\,dr-\int^T_t\,dM^n_r,\quad t\in[0,T].
\end{equation}
The proof shall be concluded by passing to the limit  in \eqref{eq5.23}. The only term  that requires care is the  integral of $f_n$. 
By \eqref{eq5.22}, there exists a subsequence (not relabelled) such that 
\[
\sup_{t\le T}|Y^n_t-Y_t|\rightarrow 0\quad  a.s.
\]
Therefore,  $X:=\sup_{n\ge 1}\sup_{t\le T}|Y^n_t|$ is finite a.s. Observe that  for any $a>0$,  \eqref{eq5.19c} holds  on the set $\{X\le a\}$. 
Thus repeating the reasoning following \eqref{eq5.19b}
we get that \eqref{eq5.19l} holds on the set $\{X\le a\}$. Since $a>0$ was arbitrary, and $X$ is finite a.s., we easily deduce that the convergence  \eqref{eq5.19l}
holds in probability $P$. Using this and \eqref{eq5.22}, and letting $n\rightarrow \infty$ in \eqref{eq5.23}, we get that $(Y,M)$ is a solution to BSDE$(\xi,f)$.
\end{proof}

\section{Proof of the main result}
\label{sec6}

\begin{lemma}
\label{lm.waog1}
Assume that \textnormal{(H1)--(H5)}  are satisfied.  Let $H^1,H^2\in \MM^{2}_0(0,T;\BR^l)$ and $(Y^i,M^i)\in\sol$ be a solution to BSDE$(\xi,f(\cdot,\cdot,H^i))$,
$i=1,2.$ Let $0\le a<b\le T$ and $(b-a)\le \frac{1}{2(1+16\lambda+\mu^+)}$. Then there exists $\beta\ge 1$ depending only on $\lambda,\mu^+,(b-a)$ such that 
\begin{equation}
\label{eq6.leqt}
\sup_{a\le t\le b}E|Y^1_t-Y^2_t|^2+E\int_a^b\,d[M^1-M^2]_r\le \beta E|Y^1_b-Y^2_b|^2+\frac14 E\int_a^T\,d[H^1-H^2]_r.
\end{equation}
\end{lemma}
\begin{proof}
By It\^o's formula
\begin{equation}\label{th3.1waog}
\begin{split}
&|Y^1_t-Y^2_t|^2+\int^b_t \,d[M^1-M^2]_r= |Y^1_b-Y^2_b|^2\\
&\quad+2\int^b_t \langle Y^1_r-Y^2_r, f(r,Y^1_r,H^1)-f(r,Y^2_r,H^2)\rangle\,dr\\
&\quad-2\int^b_t\langle Y^1_{r-}-Y^2_{r-}, d(M^1_r-M^2_r)\rangle,\quad t\in [a,b].
\end{split}
\end{equation}
By (H2) for any $\nu\ge 0$,
\begin{equation}\label{th3.2waog}
\begin{split}
&\langle Y^1_r-Y^2_r, f(r,Y^1_r,H^1)-f(r,Y^2_r,H^2)\rangle\\
&\quad \le  |Y^1_r-Y^2_r| |f(r,Y^1_r,H^1)-f(r,Y^1_r,H^2)|+\mu^+|Y^1_r-Y^2_r|^2\\
&\quad\le \frac12(1+\nu)|Y^1_r-Y^2_r|^2+\frac{1}{2(1+\nu)}|f(r,Y^1_r,H^1)-f(r,Y^1_r,H^2)|^2+\mu^+|Y^1_r-Y^2_r|^2.
\end{split}
\end{equation} 
Due to the assumptions made on $(Y^i,M^i)$ it is routine to verify that the process  $\int^\cdot_0 (Y^1_{r-}-Y^2_{r-})\,d(M^1_r-M^2_r)$  is a martingale. Therefore, applying  the expectation to (\ref{th3.1waog}) and (\ref{th3.2waog}), and combining these two inequalities, we get by (H3)
\begin{equation}\label{th3.3waogb}
\begin{split}
&E|Y^1_t-Y^2_t|^2+E\int^b_t \,d[M^1-M^2]_r\le E|Y^1_b-Y^2_b|^2\\
&\quad+\beta_\nu E\int^b_t |Y^1_r-Y^2_r|^2\,dr+\frac{\lambda}{1+\nu}E\int^T_a d[H^1-H^2]_r,
\end{split}
\end{equation}
where $\beta_\nu:= (1+\nu+2\mu^+)$. 
By  Gronwall's lemma
\begin{equation}\label{th3.4waog}
E |Y^1_t-Y^2_t|^2\le e^{\beta_\nu (b-a)}\Big(E|Y^1_b-Y^2_b|^2+\frac{\lambda}{1+\nu}E\int^T_a d[H^1-H^2]_r\Big).
\end{equation}
This combined with (\ref{th3.3waogb}) yields
\begin{align*}
E|Y^1_t-Y^2_t|^2&+E\int^b_t \,d[M^1-M^2]_r\le E|Y^1_b-Y^2_b|^2\\
&\quad+\beta_\nu(b-a) e^{\beta_\nu (b-a)}\Big(E|Y^1_b-Y^2_b|^2+\frac{\lambda}{1+\nu}E\int^T_a d[H^1-H^2]_r\Big)\\&
\quad +\frac{\lambda}{1+\nu}E\int^T_a d[H^1-H^2]_r\\&=(1+\beta_\nu(b-a) e^{\beta_\nu (b-a)})E|Y^1_b-Y^2_b|^2\\&+
\frac{\lambda}{1+\nu}(1+\beta_\nu(b-a) e^{\beta_\nu (b-a)})E\int^T_a d[H^1-H^2]_r.
\end{align*}
Set $\nu:= 16\lambda$, and $\beta:= 1+\beta_\nu(b-a) e^{\beta_\nu (b-a)}$.
Then we get the desired inequality.
\end{proof}

\begin{remark}
Observe that the last integral on the right-hand side of \eqref{eq6.leqt} is made over $[a,T]$ and not over $[a,b]$.
This is the best we can get due to the assumption (H3).
\end{remark}

\begin{theorem}\label{th3}
Assume that \textnormal{(H1)--(H5)} are satisfied. Then there exists a solution $(Y,M)$ to \textnormal{BSDE}$(\xi,f)$ such that $Y\in\mathcal{S}^{2}(0,T;\BR^l)$.
\end{theorem}
\begin{proof}
Set $(Y^0,M^0)=(0,0)$. By Proposition \ref{Proposition3} for each $n\ge 1$ there exists  a unique solution $(Y^n,M^n)\in\sol$ to BSDE$(\xi,f(\cdot,\cdot,M^{n-1}))$.
Consider a partition of the interval $[0,T]$: $0=a_1<a_2<\dots<a_{p-1}<a_p=T$ , such that 
\[
(a_{i+1}-a_i)\le \frac{1}{2(1+16\lambda+\mu^+)},\quad i=1,\dots,p-1.
\]
Set 
\[
b_{i}:=E\int_{a_{i-1}}^{a_{i}}\,d[M^2-M^1]_r,\quad i=1,\dots,p,\qquad b:=\sum_{i=1}^{p} b_i.
\]
By Lemma \ref{lm.waog1} for each  $i\in\{0,\dots, p-1\}$,
\begin{align}
\label{eq.s.waog1}
\nonumber
\sup_{a_i\le t\le a_{i+1}}E|Y^{n+1}_t-Y^n_t|^2&+E\int_{a_i}^{a_{i+1}}\,d[M^{n+1}-M^{n}]_r\\&\quad
\le \beta E|Y^{n+1}_{a_{i+1}}-Y^{n}_{a_{i+1}}|^2+\frac14 E\int_{a_i}^{T}\,d[M^n-M^{n-1}]_r.
\end{align}
In particular for any $n\ge 2$,
\begin{equation}
\label{eq.s.waog2}
\sup_{a_{p-1}\le t\le T}E|Y^{n+1}_t-Y^n_t|^2+E\int_{a_{p-1}}^{T}\,d[M^{n+1}-M^{n}]_r
\le\frac14 E\int_{a_{p-1}}^{T}\,d[M^n-M^{n-1}]_r.
\end{equation}
From this we conclude that 
\begin{equation}
\label{eq.s.waog3yvx}
\sup_{a_{p-1}\le t\le T}E|Y^{n+1}_t-Y^n_t|^2+E\int_{a_{p-1}}^{T}\,d[M^{n+1}-M^{n}]_r
\le(\frac14)^{n-1}b_{p}\le (\frac14)^{n-1}b,\quad n\ge 2.
\end{equation}
We let 
\[
I^M_n(k)=E\int_{a_{k-1}}^{a_k}\,d[M^n-M^{n-1}]_r,\quad I^Y_n(k)=\sup_{a_{k-1}\le t\le a_k}E|Y^n_t-Y^{n-1}_t|,
\]
and
\[
I_n(k)=I^M_n(k)+I^Y_n(k).
\]
Under this notation \eqref{eq.s.waog3yvx} becomes 
\begin{equation}
\label{eq.s.waog3}
I_{n+1}(p)\le (\frac14)^{n-1}b,\quad n\ge 2.
\end{equation}
Observe also  that $b_i=I^M_2(i),\, i=1,\dots,p$. By Lemma \ref{lm.waog1} and (\ref{eq.s.waog3})
\begin{align}
\label{eq.s.waog4}
\nonumber
I_{n+1}(p-1)&
\le \beta E|Y^{n+1}_{a_{p-1}}-Y^{n}_{a_{p-1}}|^2+\frac14 E\int_{a_{p-2}}^{T}\,d[M^n-M^{n-1}]_r\\&\nonumber 
\le \beta \sup_{a_{p-1}\le t\le T}E|Y^{n+1}_t-Y^n_t|^2+\frac14 E\int_{a_{p-1}}^{T}\,d[M^n-M^{n-1}]_r\\&\nonumber\quad+\frac14 E\int_{a_{p-2}}^{a_{p-1}}\,d[M^n-M^{n-1}]_r
\\&
\le(\frac14)^{n-1}b(1+\beta)+\frac14E\int_{a_{p-2}}^{a_{p-1}}\,d[M^n-M^{n-1}]_r,\quad n\ge 2.
\end{align}
Thus,
\[
I_{n+1}(p-1)\le (\frac14)^{n-1}b(1+\beta)+\frac14I^M_n(p-1),\quad n\ge 2.
\]
Consequently,
\begin{align*}
I_{n+1}(p-1)&\le (\frac14)^{n-1}b(1+\beta)+\frac14I^M_n(p-1)\le (\frac14)^{n-1}b(1+\beta)+\frac14I_n(p-1)
\\&\le (\frac14)^{n-1}b(1+\beta)+\frac14\Big((\frac14)^{n-2}b(1+\beta)+\frac14I^M_{n-1}(p-1)\Big)\\&
\le 2(\frac14)^{n-1}b(1+\beta)+\frac{1}{4^2} I^M_{n-1}(p-1).
\end{align*}
Continuing in this fashion, we get  that for each $n\ge 2$,
\begin{align}
\label{eq.s.waog5}
I_{n+1}(p-1)\le (n-1)(\frac14)^{n-1}b(\beta+1)+(\frac14)^{n-1}b_{p-1}\le n(\frac14)^{n-1}b(1+\beta).
\end{align}
For the sake of clarity of the reasoning, we shall  make  one more step.
By Lemma \ref{lm.waog1},
\begin{align}
\label{eq.s.waog465og}
\nonumber
I_{n+1}(p-2) &\le  \beta E|Y^{n+1}_{a_{p-2}}-Y^{n}_{a_{p-2}}|^2+\frac14 E\int_{a_{p-3}}^{T}\,d[M^n-M^{n-1}]_r\\&
\le \beta I^Y_{n+1}(p-1)+\frac14 \sum_{i=p-2}^p I^M_n(i),
\quad n\ge 2.
\end{align}
By \eqref{eq.s.waog3}, \eqref{eq.s.waog5} for each $n\ge 2$,
\begin{align}
\label{eq.s.waog4b8j}
\nonumber
I_{n+1}(p-2) \le n(\frac14)^{n-1}b(1+\beta)\beta+(n-1)(\frac14)^{n-1}b(1+\beta)+(\frac14)^{n-1}b+\frac14I^M_n(p-2). 
\end{align}
Hence 
\[
I_{n+1}(p-2) \le 2n(\frac14)^{n-1}b(1+\beta)^2+\frac14I^M_n(p-2),\quad n\ge 2.
\]
Therefore
\begin{align*}
I_{n+1}(p-2) &\le 2n(\frac14)^{n-1}b(1+\beta)^2+\frac14\Big(2(n-1)(\frac14)^{n-2}b(1+\beta)^2+\frac14 I^M_{n-1}(p-2)\Big)\\&
\le 2\cdot(2n)(\frac14)^{n-1}b(1+\beta)^2+\frac{1}{4^2}I^M_{n-1}(p-2).
\end{align*}
Continuing in this fashion we get  that for each $n\ge 2$,
\begin{align}
\label{eq.s.waog5250}
I_{n+1}(p-2)\le 2n^2(\frac14)^{n-1}b(1+\beta)^2.
\end{align}
Proceeding as above  for $I_{n+1}(p-i)$ with $i=3,\dots,p-1$, we get 
\begin{equation}
\label{eq.s.waog6}
I_{n+1}(k)\le pn^{p-k}(\frac14)^{n-1}b(1+\beta)^{p-k},\quad k=1,\dots,p,\quad n\ge 2.
\end{equation}
Thus,
\begin{equation}
\label{eq.s.waog7}
\sup_{0\le t\le T}E|Y^{n+1}_t-Y^n_t|^2+E\int_{0}^{T}\,d[M^{n+1}-M^{n}]_r\le p^2n^{p}(\frac14)^{n-1}b(1+\beta)^p=:q^2_n.
\end{equation}
Clearly $\sum_{n\ge 1} q_n$ is convergent. For $Y\in B_{\mathbb F}(0,T; L^{2}(\Omega,P;\BR^l))$ and $M\in\MM^{2}_0(0,T;\BR^l)$
we set 
\[
\|(Y,M)\|:=\sqrt{ \sup_{0\le t\le T}E|Y_t|^2+E\int_{0}^{T}\,d[M]_r}.
\]
It is  routine to verify  that $(B_{\mathbb F}(0,T;L^{2}(\Omega,P;\BR^l))\times \MM^{2}_0(0,T;\BR^l),\|\cdot\|)$
is a Banach space. Observe that for $n<m$,
\begin{equation}
\label{eq3.cs.waog1}
\|(Y^m,M^m)-(Y^n,M^n)\|\le \sum_{i=n}^{m-1}q_i\le \sum_{i=n}^{\infty}q_i.
\end{equation}
Thus, $(Y^n,M^n)_{n\ge 1}$ is a Cauchy sequence in  $B_{\mathbb F}(0,T;L^{2}(\Omega,P;\BR^l))\times \MM^{2}_0(0,T;\BR^l)$. By It\^o's formula
\begin{equation*}
\begin{split}
|Y^m_t-Y^n_t|^2&+\int_t^T\,d[M^m-M^n]_r\\&
=2\int_t^T\langle f(r,Y^m_r,M^{m-1})-f(r,Y^n_r,M^{n-1}), Y^m_r-Y^n_r\rangle\,dr\\&
-2\int_t^T\langle Y^m_{r-}-Y^n_{r-}, d(M^m_r-M^n_r)\rangle,\quad t\in [0,T].
\end{split}
\end{equation*}
By  (H2)
\begin{align}
\label{eq6.45}
\nonumber\langle f(r,Y^m_r,M^{m-1})&-f(r,Y^n_r,M^{n-1}), Y^m_r-Y^n_r\rangle\\&
\le (1+\mu^+)|Y^m_r-Y^n_r|^2+|f(r,Y^m_r,M^{m-1})-f(r,Y^m_r,M^{n-1})|^2.
\end{align}
This combined with the previous inequality yields
\begin{equation*}
\begin{split}
E\sup_{0\le t\le T}|Y^m_t-Y^n_t|^2&+E\int_0^T\,d[M^m-M^n]_r\\&\quad
\le 4(1+\mu^+)E\int_0^T|Y^m_r-Y^n_r|^2\,dr \\&\quad+4 E\int_0^T|f(r,Y^m_r,M^{m-1})-f(r,Y^m_r,M^{n-1})|^2\,dr\\&\quad
+4E\sup_{0\le t\le T}\Big|\int_t^T(Y^m_{r-}-Y^n_{r-})\,d(M^m_r-M^n_r)\Big|.
\end{split}
\end{equation*}
By (H3) and the Burkholder-Davis-Gundy inequality there exists $C>0$ such that
\begin{equation*}
\begin{split}
E\sup_{0\le t\le T}|Y^m_t&-Y^n_t|^2+E\int_0^T\,d[M^m-M^n]_r\\&
\le 4(1+\mu^+)E\int_0^T|Y^m_r-Y^n_r|^2\,dr +4 \lambda E\int_0^T\,d[M^{m-1}-M^{n-1}]_r\\&
+4C E\Big(\int_0^T|Y^m_{r-}-Y^n_{r-}|^2\,d[M^m_r-M^n_r]\Big)^{\frac12}\\&
\le 4(1+\mu^+)E\int_0^T|Y^m_r-Y^n_r|^2\,dr +4 \lambda E\int_0^T\,d[M^{m-1}-M^{n-1}]_r\\&
+\frac12E\sup_{0\le t\le T}|Y^m_t-Y^n_t|^2+8C^2E\int_0^T\,d[M^m-M^n]_r.
\end{split}
\end{equation*}
Hence, 
\begin{equation*}
\begin{split}
E\sup_{0\le t\le T}|Y^m_t&-Y^n_t|^2+E\int_0^T\,d[M^m-M^n]_r\\&
\le 8(1+\mu^+)E\int_0^T|Y^m_r-Y^n_r|^2\,dr +8 \lambda  E\int_0^T\,d[M^{m-1}-M^{n-1}]_r\\&
+16C^2E\int_0^T\,d[M^m-M^n]_r\\&
\le(8\mu^+T+8T+16C^2)\|(Y^m,M^m)-(Y^n,M^n)\|^2\\&
+8 \lambda \|(Y^{m-1},M^{m-1})-(Y^{n-1},M^{n-1})\|^2.
\end{split}
\end{equation*}
From this and (\ref{eq3.cs.waog1}) we conclude that $(Y^n,M^n)_{n\ge 1}$ is a Cauchy sequence in the space $\sol$.
Therefore, there exists a pair of c\`adl\`ag processes $(Y,M)\in \sol$ such that
\begin{equation}
\label{eq3.waog.14}
E\sup_{0\le t\le T}|Y^n_t-Y_t|^2+E\int_0^T\,d[M^n-M]_r\rightarrow 0.
\end{equation}
By the definition of a solution to BSDE$(\xi,f(\cdot,\cdot,M^{n-1}))$,
\begin{equation}
\label{eq3.mpie.waog1}
Y^n_t=\xi+\int_t^Tf(r,Y^n_r,M^{n-1})\,dr-\int_t^T\,dM^n_r,\quad t\in [0,T]
\end{equation}
The final step  in the proof thus shall be  passing  to the limit in (\ref{eq3.mpie.waog1}). For this we have to cope with the nonlinear term of the equation.
Clearly,
\begin{align}
\label{eq3.waog.16}
\nonumber
|f(r,Y^n_r,M^{n-1})-f(r,Y_r,M)|&\le|f(r,Y^n_r,M)-f(r,Y^n_r,M^{n-1})|\\&\quad
+|f(r,Y^n_r,M)-f(r,Y_r,M)|.
\end{align}
By (H3) and (\ref{eq3.waog.14})
\begin{equation}
\label{eq3.waog.15}
\begin{split}
E\int_0^T |f(r,Y^n_r,M)-f(r,Y^n_r,M^{n-1})|\,dr
\le  \sqrt \lambda \sqrt T\Big(E\int_0^T \,d[M^{n-1}-M]_r\Big)^{\frac12} \rightarrow 0.
\end{split}
\end{equation}
Now we shall focus on the second term on the right-hand side of (\ref{eq3.waog.16}). Let $(n_k)$ be a subsequence of $(n)$.
Then by (\ref{eq3.waog.14}) there exists a further subsequence $(n_{k_l})$ such that
\begin{equation}
\label{eq3.waog.17}
\sum_{l=1}^\infty E\sup_{0\le t\le T}|Y^{n_{k_l}}_t-Y_t|<\infty.
\end{equation}
Set $l:= n_{k_l}$. For $\nu>0$ set 
\[
A_\nu:=\{\omega\in\Omega: \exists_{t\in [0,T]}\,\,\exists_{l\ge 1}\,\, |Y^l_t(\omega)|\ge \nu\}.
\]
We have 
\begin{equation}
\label{eq3.waog.18}
\begin{split}
P(\int_0^T|f(r,Y^l_r,M)&-f(r,Y_r,M)|\,dr >\varepsilon)\le P(A_\nu)\\&
+P(\mathbf{1}_{A^c_\nu}\int_0^T |f(r,Y^l_r,M)-f(r,Y_r,M)|\,dr>\varepsilon).
\end{split}
\end{equation}
By (\ref{eq3.waog.17})
\begin{equation}
\label{eq3.waog.19}
\begin{split}
P(A_\nu)&=P(\exists_{l\ge 1}\,\, \sup_{0\le t\le T}|Y^l_t|\ge \nu)\\&
\le P(\exists_{l\ge 1}\,\, \sup_{0\le t\le T}|Y^l_t-Y_t|\ge \nu/2)+ P(\sup_{0\le t\le T}|Y_t|\ge  \nu/2)\\&\le
\frac2\nu E\sup_{0\le t\le T}|Y_t|+\frac{2}{\nu} \sum_{l=1}^\infty E\sup_{0\le t\le T}|Y^l_t-Y_t|\rightarrow 0,\quad \nu\rightarrow \infty. 
\end{split}
\end{equation}
By (\ref{eq3.waog.17}) and (H4)
\begin{equation}
\label{eq3.waog.20}
|f(r,Y^l_r,M)-f(r,Y_r,M)|\rightarrow 0,\quad l\rightarrow \infty,\quad \ell ^1\otimes P\mbox{-a.e.}
\end{equation}
Observe that 
\begin{align*}
\mathbf{1}_{A^c_\nu}|f(r,Y^l_r,M)-f(r,Y_r,M)|&\le |f(r,Y_r,M)-f(r,Y_r,0)|+ |f(r,Y^l_r,M)-f(r,Y^l_r,0)|\\&\quad
+\mathbf{1}_{A^c_\nu}|f(r,Y^l_r,0)|+\mathbf{1}_{A^c_\nu}|f(r,Y_r,0)|\\&\le  |f(r,Y_r,M)-f(r,Y_r,0)|+ |f(r,Y^l_r,M)-f(r,Y^l_r,0)|\\&\quad
+2\sup_{|y|\le \nu}|f(r,y,0)|=:g^\nu_l(r).
\end{align*}
By  (H3) and (H5), the family $(g^\nu_l)_{l\ge 1}$ is uniformly integrable with respect to the measure $\ell^1\otimes P$.
From this, (\ref{eq3.waog.20}), and the Vitali  convergence theorem,
we conclude that the second term on the right-hand side of  (\ref{eq3.waog.18}) goes to zero as $l\rightarrow \infty$.
Therefore, letting $l\rightarrow \infty$ in (\ref{eq3.waog.18}) we get
\[
\limsup_{l\rightarrow \infty}P(\int_0^T|f(r,Y^l_r,M)-f(r,Y_r,M)|\,dr >\varepsilon)\le P(A_\nu).
\]
Then, letting $\nu\rightarrow \infty$ and using (\ref{eq3.waog.19}) we obtain 
\[
\int_0^T|f(r,Y^l_r,M)-f(r,Y_r,M)|\,dr\rightarrow 0,\quad l\rightarrow \infty
\]
in probability $P$. Since $(n_k)$ was an arbitrary subsequence of $(n)$, we get that the above convergence holds with
$l$ replaced by $n$. This combined with (\ref{eq3.waog.16}), (\ref{eq3.waog.15}) implies 
\[
\sup_{0\le t\le T}\Big|\int_t^Tf(r,Y^n_r,M^{n-1})-\int_t^T f(r,Y_r,M)\,dr\Big|\rightarrow 0,\quad l\rightarrow \infty
\]
in probability $P$. By using this convergence  and (\ref{eq3.waog.14}), we let $n\rightarrow \infty$ in (\ref{eq3.mpie.waog1})
to get
\[
Y_t=\xi+\int_t^Tf(r,Y_r,M)\,dr-\int_t^T\,dM_r,\quad t\in [0,T].
\]
Thus, $(Y,M)$ is a solution to BSDE$(\xi,f)$. 
\end{proof}

\section{Lipschitz type condition on the driver}
\label{sec7}
The purpose of this section is to demonstrate that the method presented in Section \ref{sec6}
also works for generators of the form
\[
f:[0,T]\times\Omega\times\sol\to \BR^l.
\]
As  usual  we assume that for any $(Y,M)\in\sol$ process  $f(\cdot,Y,M)$ is progressively measurable.

\begin{definition}\label{def2}
We say that a pair $(Y,M)\in\sol$  is a solution of BSDE($\xi,f$) if
\begin{enumerate}
\item[(a)] $Y$ is c\`adl\`ag,
\item[(b)] $\int^T_0|f(r,Y,M)|\,dr<+\infty$,
\item[(c)] $Y_t=\xi+\int^T_t f(r,Y,M)\,dr-\int^T_t\,dM_r,\quad t\in[0,T].$
\end{enumerate}
\end{definition}

As to the regularity of $f$ with respect to $Y$-variable, we consider  an analogue 
of condition (H3). We formulate it as a one condition which expresses regularity of $f$ with respect to $Y$ and $M$-variable.

\begin{enumerate}
\item[(A)] There exists $L>0$ such that for any $t\in [0,T]$, $M,M'\in \MM^2_0(0,T;\BR^l)$, $Y,Y'\in \mathcal S^2(0,T;\BR^l)$
\[
E\int_t^T|f(r,Y,M)-f(r,Y',M')|^2\,dr\le L(E\int_t^T\,d[M-M']_r+E\int_t^T|Y_r-Y'_r|^2\,dr).
\]
\end{enumerate}

This is thus a type of Lipschitz continuity, so in some sense stronger than (H2), however it permits $f(t,Y,M)$
to depend  not only on the future of the process $M$ but also on the future of the process $Y$.
This generalization allows us in particular to consider in place of $Y_t$ in Examples \ref{ex.1}--\ref{ex.6}, $Y_{t+\delta(t)}$
or $\mathcal L(Y_t)$.  An  existence result   for \eqref{eqi.101} with $f$ satisfying condition of type (A)  was provided in \cite{CN}.
The authors, however, assumed in \cite{CN} that $f$ is of the special form \eqref{eqi.0} with $h$ defined in Example \ref{ex.3}.

\begin{theorem}
Assume that (H1) and (A) are in force. Then there exists a unique solution to BSDE$(\xi,f)$.
\end{theorem}
\begin{proof}
\label{th7.1}
It is enough to repeat the arguments of Section \ref{sec6} with small modifications.
The proof of Lemma \ref{lm.waog1} may  be repeated step by step with one exception that 
instead of \eqref{th3.2waog} one should apply
\begin{equation}\label{eq7.1}
\begin{split}
&E\int_t^T\langle Y^1_r-Y^2_r, f(r,Y^1,H^1)-f(r,Y^2,H^2)\rangle\,dr\\
&\quad\le \frac12(1+\nu+L)E\int_t^T|Y^1_r-Y^2_r|^2\,dr+\frac{L}{2(1+\nu)}E\int_t^T\,d[H^1-H^2]_r,
\end{split}
\end{equation} 
which follows directly from (A). The whole proof of Theorem \ref{th3} up to \eqref{eq3.mpie.waog1}
may be repeated under assumptions (H1), (A) with  one exception that instead of \eqref{eq6.45}
one should employ the following inequality
\begin{align}
\label{eq7.2}
\nonumber E\int_0^T|\langle f(r,Y^m,M^{m-1})&-f(r,Y^n,M^{n-1}),Y^m_r-Y^n_r\rangle|\,dr \\&
\le (1+L)E\int_0^T|Y^m_r-Y^n_r|^2\,dr+L E\int_0^T\,d[M^{m-1}-M^{n-1}]_r,
\end{align}
which once again follows directly  from (A). Here note that the existence of $(Y^n,M^n)$
follows from \cite[Proposition 3.7]{CN}.  Now we can easily pass to the limit  in \eqref{eq3.mpie.waog1}
by using condition (A) again.
\end{proof}

\subsection*{Acknowledgements}
{\small  T. Klimsiak is supported by Polish National Science Centre: Grant No. 2016/23/B/ST1/01543.  M. Rzymowski  acknowledges the support of the  Polish  National Science Centre: Grant No.  2018/31/N/ST1/00417.}


\begin{thebibliography}{99}

\bibitem{BPS}
Bally, V., Pardoux, \'E., Stoica, L.: Backward stochastic differential equations associated to a
symmetric Markov process. {\em Potential Anal.} {\bf 22} (2005) 17--60.


\bibitem{BBP}
Barles, G.,  Buckdahn, R.,  Pardoux, \'E.:
Backward stochastic differential equations and integral-partial differential equations. 
{\em Stochastics Stochastics Rep.} {\bf 60} (1997)  57--83.

\bibitem{BR}
Barrasso, A., Russo, F.: Decoupled mild solutions of path-dependent PDEs and integro PDEs represented by BSDEs driven by cadlag martingales. {\em Potential Anal.} {\bf 53} (2020)  449--481.

\bibitem{BLY}
Bensoussan, A., Li, Y., Yam, S.: Backward stochastic dynamics with a subdifferential operator and non-local parabolic variational inequalities. 
{\em Stoch. Anal. Appl.} {\bf 128} (2018)  644--688.

\bibitem{BSW}
B\"ottcher, B., Schilling, R., Wang, J.: L\'evy Matters III. L\'evy-Type Processes: Construction,
Approximation and Sample Path Properties. {\em Lecture Notes in Math.} {\bf 2099}
Springer, Cham (2013).









\bibitem{BDHPS}
Briand, Ph., Delyon, B., Hu, Y., Pardoux, \'E., Stoica, L.: $L^{p}$ solutions of Backward Stochastic Differential Equations. {\em Stochastic Process. Appl.} {\bf 108} (2003) 109--129.




\bibitem{BLP}
Buckdahn, R., Li, J., Peng, S.:
Mean-field backward stochastic differential equations
and related partial differential equations.
{\em Stochastic Process. Appl.} {\bf 119} (2009) 3133--3154.



\bibitem{CN}
Cheridito, P., Nam, K.:
BSE’s, BSDE’s and fixed point problems.  
{\em Ann. Probab.} {\bf 45} (2017) 3795--3828.






\bibitem{CE}
Cohen, S.N.,  Elliott, R.J.: Existence, uniqueness and comparisons for BSDEs in general spaces.{\em  Ann. Probab.} {\bf 40}
(2012) 2264--2297.




\bibitem{JS}
Jacod, J., Shiryaev, A.N.: Limit Theorems for Stochastic Processes, 2nd edn. Springer, Berlin
(2003).


\bibitem{K:JFA}
Klimsiak, T.: Semi-Dirichlet forms, Feynman-Kac functionals and the
Cauchy problem for semilinear parabolic equations. {\em J. Funct.
Anal.} {\bf 268} (2015) 1205--1240.

\bibitem{K:SPA}
Klimsiak, T.: Reflected BSDEs on filtered probability spaces. {\em
Stochastic Process. Appl.} {\bf 125} (2015) 4204--4241.

\bibitem{KP}
Kruse, T., Popier, A.: BSDEs with monotone generator driven by Brownian and Poisson noises in a general filtration. {\em Stochastics} {\bf 88} (4) (2016) 491--539.

\bibitem{LLQ}
Liang, G., Lyons, T., Qian Z.: Backward stochastic dynamics on a filtered probability space.
{\em Ann. Probab.} {\bf 39} (2011) 1422--1448.





\bibitem{NS}
Nualart, D., Schoutens, W.: Backward stochastic differential equations and Feynman-Kac formula for L\'evy processes, with applications in finance. {\em Bernoulli} {\bf 7} (5) (2001) 761--776.

\bibitem{Pardoux}
Pardoux, \'E.:  Backward stochastic differential equations and viscosity solutions of systems of
semilinear parabolic and elliptic PDEs of second order. In Stochastic analysis and related topics, VI
(Geilo, 1996) {\em Progr. Probab.} {\bf 42} Birkhäuser Boston, Boston, MA (1998) 79--127.



\bibitem{Pardoux1}
Pardoux, \'E.: BSDEs, weak convergence and homogenization of
semi-linear PDEs. {\em Proc. S\'eminaires de Math\'ematiques
Sup\'erieures (Montr\'eal, 1998)} F.H. Clarke and
R.J. Stern (Eds.) Kluwer, Dordrecht (1999) 503--549.

\bibitem{PP}
Pardoux, \'E., Peng, S.: Adapted solution of a backward stochastic differential equation. {\em Systems
Control Lett.} {\bf 14}(1) (1990) 55--61.

\bibitem{Peng}
Peng, S.: Nonlinear expectations, nonlinear evaluations and risk measures. {\em Stochastic
Methods in Finance. Lecture Notes in Math.} {\bf 1856} Springer, Berlin (2004) 165--253.

\bibitem{PY}
Peng, S., Yang, Z.: Anticipated backward stochastic differential equations. {\em Ann.
Probab.} {\bf 37} (2009) 877--902.

\bibitem{QS}
Quenez, M.-C., Sulem, A.:
BSDEs with jumps, optimization and applications to dynamic risk measures. 
{\em Stochastic Process. Appl.} {\bf 123}  (2013) 3328--3357.




\bibitem{R:PTRF}
Rozkosz, A.:  Backward SDEs and Cauchy problem for semilinear equations in divergence form. {\em Probab. Theory Related Fields} {\bf 125} (2003) 393--407.


\bibitem{Royer}
Royer, M.: Backward stochastic differential equations with jumps
and related non-linear expectations. {\em Stochastic Process. Appl.} {\bf 116} (2006) 1358--1376.




\bibitem{Situ}
Situ, R.: On solution of backward stochastic differential equations with jump and applications. {\em Stochastic Process. Appl.} {\bf 66} (1997) 209-236.



\bibitem{YM}
Yin, J., Mao, X.:
The adapted solution and comparison theorem for backward stochastic differential equations with Poisson jumps and applications. {\em J. Math. Anal. Appl.} {\bf 346}  (2008) 345--358.




\end{thebibliography}
\end{document}